\documentclass[bj,preprint,12pt]{imsart}

\RequirePackage[OT1]{fontenc}
\RequirePackage{amsthm,amsmath}
\RequirePackage[numbers]{natbib}
\RequirePackage[colorlinks,citecolor=blue,urlcolor=blue]{hyperref}

\startlocaldefs
\numberwithin{equation}{section}
\theoremstyle{plain}

\endlocaldefs

\newcommand{\ignore}[1]{}

\usepackage{natbib}
\usepackage{multirow}
\usepackage{graphicx}
\usepackage{morefloats}
\usepackage{subfigure}
\usepackage{epsfig,graphics,rotating}
\usepackage{wrapfig}
\usepackage{savesym}
\usepackage{hyperref}
\usepackage{enumerate}
\usepackage{geometry}
\usepackage{amsmath,amssymb,amsthm}
\numberwithin{equation}{section}

\newtheorem{theorem}{Theorem}[section]
\newtheorem{lemma}[theorem]{Lemma}
\newtheorem{proposition}[theorem]{Proposition}

\newtheorem{remark}[theorem]{Remark}

\newcommand*{\tran}{^{\mkern-1.5mu\mathsf{T}}}

\begin{document}
\title{Optimal Gaussian Approximation For Multiple Time Series}
\author{Sayar Karmakar$^{\dagger,*}$ and Wei Biao Wu$^{*}$}

\maketitle

\begin{center}
 \textit{$^{\dagger}$University of Florida and $^{*}$University of Chicago}
\end{center}

{\it Abstract.} We obtain an optimal bound for a Gaussian approximation of a large class of vector-valued random processes. Our results provide a substantial generalization of earlier results that assume independence and/or stationarity. Based on the decay rate of the functional dependence measure, we quantify the error bound of the Gaussian approximation using the sample size $n$ and the moment condition. Under the assumption of $p$th finite moment, with $p>2$, this can range from a worst case rate of $n^{1/2}$ to the best case rate of $n^{1/p}$.


\section{Introduction}
The functional central limit theorem (FCLT), or invariance principle plays an important role in statistics. Let $X_i$ for $i \ge 1,$ be independent and identically distributed (i.i.d.) random vectors in $\mathbb{R}^d$ with mean zero and covariance matrix $\Sigma$, and let $S_j = \sum_{i=1}^j X_i$. The FCLT asserts that
\begin{equation}\label{eq:fclt}
    \{ n^{-1/2} S_{\lfloor n u \rfloor}, \, 0 \le u \le 1\} \Rightarrow \{\Sigma^{1/2} I\!B(u) , \, 0 \le u \le 1\},
\end{equation}
where $\lfloor t \rfloor = \max\{i \in \mathbb Z: i \le t\}$ and $I\!B$ is the standard Brownian motion in $\mathbb{R}^d$; that is it has independent increments, and $I\!B(u+v) - I\!B(u) \sim N(0, v I_d)$ for $u, v \ge 0$. In this study, we generalize (\ref{eq:fclt}) by developing a convergence rate of (\ref{eq:fclt}) for multiple time series that can be dependent and nonidentically distributed.

The invariance principle was introduced by Erd{\"o}s and Kac (1946, \cite{MR0015705}). Doob (1949, \cite{MR0030732}), Donsker (1952, \cite{MR0047288}), and Prohorov (1956, \cite{MR0084896}) furthered their ideas, which led to the theory of weak convergence of probability measures. There is an extensive body of literature on Gaussian approximations when the dimension $d = 1$. In this case, optimal rates for independent random variables were obtained by \cite{MR0375412} and \cite{MR2302850}, among others. When $d = 1$ and $X_i$ is i.i.d. with mean zero and variance $\sigma^2$ and has a finite $p$th moment for $p > 2$, Koml{\'o}s, Major, and Tusn{\'a}dy (1975, 76, \cite{MR0375412, MR0402883}) established the following result: 
\begin{eqnarray}\label{eq:errorrate}
\max_{1\leq i \leq n} |S_i'- \sigma B(i)| = o_{\rm a.s.}(\tau_n),
\end{eqnarray}
where $B(\cdot)$ is the standard Brownian motion and $S_n'$ is constructed on a richer space; such that $(S_i)_{i \leq n}\stackrel{D}{=}(S_i')_{i \leq n}$, and the approximation rate $\tau_n = n^{1/p}$ is optimal. Results of the type shown in (\ref{eq:errorrate}) have many applications in statistics because we can use functionals involving  Gaussian processes to approximate statistics of $(X_i)_{i=1}^n$, and thus exploit the properties of Gaussian processes. Their result was generalized to independent random vectors by Einmahl (1987a, \cite{MR899446}; 1987b, \cite{MR905340}; 1989, \cite{MR996984}), Zaitsev (2001, \cite{MR1968723}; 2002a, \cite{MR1971831}; 2002b, \cite{MR1978667}), and G{\"o}tze and Zaitsev (2008, \cite{MR2760567}), who optimal and nearly optimal results. 

To generalize (\ref{eq:errorrate}) to multiple time series, we consider the possibly nonstationary, $d$-dimensional, mean zero, vector-valued process
\begin{eqnarray}\label{eq:representation}
X_i = (X_{i1},\ldots, X_{id}) \tran = H_i (\mathcal{F}_i)= H_i ( \epsilon_{i}, \epsilon_{i-1} ,\ldots ), \quad i \in \mathbb Z,
\end{eqnarray}
where $\tran$ denotes a matrix transpose, $\mathcal{F}_i = ( \epsilon_i, \epsilon_{i-1}, \ldots )$ and $\epsilon_i$ for $i \in \mathbb Z,$ are i.i.d. random variables. Here, $H_i(\cdot)$ is a measurable function such that $X_i$ is well defined. We allow $H_i$ to be possibly nonlinear in its argument $(\epsilon_i,\epsilon_{i-1},\ldots )$ in order to capture a much larger class of processes. If $H_i(\cdot) \equiv H(\cdot)$ does not depend on $i$, (\ref{eq:representation}) defines a stationary causal process. The latter framework is very general; see \cite{MR2778591, MR2172215, MR991969}, among others. When $d = 1$, \citet{wiener1958nonlinear} considered representing stationary processes by functionals of i.i.d. random variables.

\citet{MR2172368} presented numerous applications of the functional central limit theorem for multiple time series analysis. Wu and Zhao (2007, \cite{MR2323759}) and Zhou and Wu (2010, \cite{MR2758526}) applied Gaussian approximation results with suboptimal approximation rates to trend estimations and functional regression models. For the class of weakly dependent processes (\ref{eq:representation}), we show that there exists a probability space $(\Omega_c, A_c, P_c)$ on which we can define random vectors $X_i^c$, with the partial sum process $S_i^c= \sum_{t=1}^i X_t^c$ and a Gaussian process $G_i^c= \sum_{t=1}^i Y_t^c$. Here $Y_t^c$ is a mean zero independent Gaussian vector, such that $(S_i^c)_{1 \le i \le n} \stackrel{D}{=} (S_i)_{1 \le i \le n}$ and 
\begin{eqnarray}\label{eq:Main Gaussian approximation eq}
\displaystyle\max_{i \leq n}|S_i^c- G_i^c| = o_P(\tau_n) \quad \text{in } (\Omega_c,A_c, P_c),
\end{eqnarray} 
where the approximation bound $\tau_n$ is related to the dependence decaying rates. Our result is useful for asymptotic inferences involving multiple time series. As a primary contribution, we generalize and improve the existing results for Gaussian approximations in several ways. For some $p>2$, we assume uniform integrability of the $p$th moment and obtain an approximation bound $\tau_n$ in terms of $p$ and the decay rate of the functional dependence measure. In particular, if the dependence decays sufficiently quickly, for $\tau_n$, we are able to achieve the optimal $o_P(n^{1/p})$ bound. In the current literature, optimal results have been obtained for some special cases only. We start with a brief overview of these.

For stationary processes with $d = 1$, a suboptimal rate was derived by Wu (2007, \cite{MR2353389}), where the martingale approximation is applied. Berkes, Liu, and Wu (2014, \cite{MR3178474}) considered the causal stationary process given in  (\ref{eq:representation}) above obtaining the $n^{1/p}$ bound for $p>2$. It is considerably more challenging to deal with vector-valued processes. Eberlein (1986, \cite{MR815969}) obtained a Gaussian approximation result for dependent random vectors with an approximation error $O(n^{1/2-\kappa} ),$ for some small $\kappa > 0$. However, this bound can be too crude for many statistical applications. The martingale approximation approach in \cite{MR2353389} cannot be applied to vector-valued processes because Strassen's embedding fails for vector-valued martingales \cite{monrad1991problem} in general. For a stationary multiple time series with additional constraints, Liu and Lin (2009, \cite{MR2485027}) obtained an important result on strong invariance principles for stationary processes with bounds of the order $n^{1/p}$, with $2<p<4$.  Wu and Zhou (2011, \cite{MR2827528}) obtained suboptimal rates for multiple nonstationary time series. A critical limitation of the results in \cite{MR2827528, MR2485027} is the restriction $2<p<4$. Whether the bound $n^{1/p}$ can be achieved when $p \ge 4$ remains an open problem.

In this paper, we show that under proper decaying conditions on functional dependence measures for the process (\ref{eq:representation}), we can indeed obtain the optimal bound $n^{1/p}$ for $p \ge 4$. Our condition is stated in the form of (\ref{eq:form of thetaip}), which employs the two parameters $\chi$ and $A$ to formulate the temporal dependence of the process. In general, larger values of $\chi$ and $A$ mean the dependence decays more quickly. With proper conditions on $A$, we find optimal $\tau_n = \tau_n(\chi)$ for a general $\chi>0$. In Corollary 2.1 in Berkes, Liu, and Wu (2014, \cite{MR3178474}) the authors discussed univariate and stationary processes. However, their focus was on larger values of $\chi$ that allowed them to obtain $\tau_n=n^{1/p}$. In Theorem \ref{th:main theorem}, we obtain a rate for any $\chi>0$, and show that if $\chi$ increases from 0 to a certain number $\chi_0$, we obtain the optimal $\tau_n$, varying from the worst, $n^{1/2}$, to the optimal, $n^{1/p}$. This work is useful for processes in which dependence does not decay sufficiently quickly. For the borderline case $\chi=\chi_0$, we have a rate of $o_P(n^{1/p})$ for $2<p<4$, and for $p \geq 4$, we have a rate of $o_P(n^{1/p} \log n)$. However, if $\chi>\chi_0$, we obtain the optimal $o_P(n^{1/p})$ bound for all $p>2$.

Our sharp Gaussian approximation result is quite useful for simultaneous inferences of curves where the unknown function is not even Lipschitz continuous. Although many studies have examined curve estimations by assuming smooth or regular behavior of a function few have focused on functions that are not differentiable or not Lipschitz continuous. Our Gaussian approximation can play a key role in weakening the smoothness assumption and thus enlarging the scope of statistical inferences.  Moreover, the optimal $o_P(n^{1/p})$ bound for $2<p<4$ and the stationary processes obtained in \cite{MR2485027} have remained popular choices over the past few years for multivariate Gaussian approximations. Therefore, we can apply our sharper invariance principle to generalize that of (\cite{MR2485027}) one in multiple ways, thus yielding optimal rates when $p \ge 4$.

The rest of the article is organized as follows. In section \ref{sec:mr}, we introduce the functional dependence measure and present our main result. Applications to linear processes and to locally stationary nonlinear nonLipschitz processes are given in section \ref{sec:appl}. The proof of Theorem \ref{th:main theorem} 
is outlined in section \ref{sec:proof1}. A detailed version is provided in the online Supplementary Material section \ref{sec:proofdetails}. The goal of the sketched outline is to give the readers a basic idea of our long and involved derivation. Some useful results used throughout the proofs are presented in the online Supplementary Material section \ref{sec:sul}.

We now introduce some notation. For a random vector $Y$, write $Y \in \mathcal{L}_p$, for $p > 0$, if $\|Y \|_p := E(|Y |^p ) ^{1/p} < \infty$. If $Y \in \mathcal{L}_2$, $Var(Y)$ denotes the covariance matrix. For the $\mathcal{L}_2$ norm write $\|\cdot \| =\| \cdot \|_ 2$. Throughout the text, $c_p$ denotes a constant that depends only on $p$ and $c$ denotes a universal constants. These might take different values in different lines, unless otherwise specified. Then, $x^+=\max(x,0)$ and $x^{-}=-\min(x,0)$. For two positive sequences $a_n$ and $b_n$, if $a_n/b_n \to 0$ (resp. $a_n/b_n \to \infty$), write $a_n \ll b_n$ (resp. $a_n \gg b_n$). Write $a_n \lesssim b_n$ if $a_n \leq c b_n$, for some $c<\infty$. The $d$-variate normal distribution with mean $\mu$ and covariance matrix $\Sigma$ is denoted by $N(\mu, \Sigma)$. Denote by $I_d$ the $d \times d$ identity matrix. For a matrix $A= (a_{ij})$, we define its Frobenius norm as $|A|= (\sum a_{ij}^2)^{1/2}$. For a positive semi-definite matrix $A$ with spectral decomposition $A= Q D Q\tran$, where $Q$ is orthonormal and $D = (\lambda_1, \ldots, \lambda_d)$ with $\lambda_1 \ge \ldots \ge \lambda_d$, write the Grammian square root as $A^{1/2}= Q D^{1/2} Q\tran$, where $\rho_*(A) = \lambda_d$ and $\rho^*(A) = \lambda_1$.

\section{Main Results}
\label{sec:mr}
We first introduce the uniform functional dependence measure on the underlying process using the idea of coupling. Let $\epsilon_i', \epsilon_j$, for $i, j \in \mathbb{Z},$ be i.i.d. random variables. Assume $X_i \in \mathcal{L}_p, p>0$. For $j \geq 0$, $0<r\leq p$, define the functional dependence measure 
\begin{eqnarray}\label{eq:fdm}
\delta_{j,r}= \sup_{i} \| X_i - X_{i,(i-j)} \|_r = \sup_{i} \|H_i(\mathcal{F}_i)- H_i(\mathcal{F}_{i, (i-j)}) \|_r ,
\end{eqnarray}
\noindent where $\mathcal{F}_{i,(k)}$ is the coupled version of $\mathcal{F}_i$, with $\epsilon_k$ in $\mathcal{F}_i$ replaced by an i.i.d.  copy $\epsilon_k'$,
\begin{eqnarray*}\label{eq:F_ik}
\mathcal{F}_{i,(k)}= (\epsilon_i, \epsilon_{i-1},  \ldots, \epsilon_k', \epsilon_{k-1}, \ldots )
 \mbox{ and } X_{i,(i-j)}= H_i(\mathcal{F}_{i,(i-j)}).
\end{eqnarray*}
In addition, $\mathcal{F}_{i,(k)}= \mathcal{F}_i$ if $k>i$. Note that, $\|H_i(\mathcal{F}_i)- H_i (\mathcal{F}_{i, (i-j)}) \| _r$ measures the dependence of $X_i$ on $\epsilon_{i-j}$. Because the physical mechanism function $H_i$ may differ for a nonstationary process, we choose to define the functional dependence measure in a uniform manner. The quantity $\delta_{j,r}$ measures the uniform $j$-lag dependence in terms of the $r$th moment. Assume throughout that
\begin{eqnarray}
\Theta_{0,p}=\displaystyle\sum_{i=0}^{\infty}\delta_{i,p} <\infty.
\end{eqnarray}
This condition implies short-range dependence in the sense that the cumulative dependence of $(X_j)_{j \geq k}$ on  $\epsilon_k$ is finite. For clarity of presentation, in this paper we assume there exists $\chi > 0, A > 0$ such that the tail cumulative dependence measure
\begin{eqnarray}
 \label{eq:form of thetaip}
\Theta_{i,p} = \sum_{j=i}^{\infty}\delta_{j,p} = O\left(i^{-\chi} (\log i)^{-A}\right).
\end{eqnarray}
Larger $\chi$ or $A$ implies weaker dependence. Our Gaussian approximation rate $\tau_n$ (cf., Theorems \ref{th:main theorem} and \ref{th:theorem 2}) depends on $\chi$ and $A$. Define functions $f_j(\cdot, \cdot)$ as follows
\begin{eqnarray}\label{eq:fs}
f_1 &=& f_1(p,\chi)=p^2 \chi^2+p^2\chi, \,\, f_2 = 2p \chi^2+3p \chi- 2\chi, \\ \nonumber 
f_3 &=& p^3(1+\chi)^2 + 6 f_1+ 4p\chi-2, \,\, f_4 = 2 p(2p\chi^2+3p\chi+p-2), \\ \nonumber 
\quad f_5 &=& p^2( p^2+4p-12)\chi^2+2p(p^3+p^2-4p-4)\chi+ (p^2-p-2)^2.
\end{eqnarray}

\noindent Assume that the process in (\ref{eq:representation}) satisfies the uniform integrability and regularity conditions on the covariance structure:
\begin{enumerate}[(2.A)]
\item The series $(|X_i|^p)_{i \ge 1}$ is uniformly integrable: $\sup_{i \ge 1} E (|X_i|^p {\bf 1}_{|X_i| \ge u}) \to 0 \mbox{ as } u \to \infty;$
\item (Lower bound on eigenvalues of covariance matrices of increment processes) There exists $\lambda_* > 0$ and $l_* \in \mathbb N$, such that for all $t \ge 1, l \ge l_*$,  $$\rho_*(Var (S_{t+l}- S_t)) \geq \lambda_* l.$$
\end{enumerate}
The uniform integrability assumption is necessary owing to the nonstationarity of the process. The latter is frequently imposed in study of multiple time series.

\begin{theorem}
\label{th:main theorem}
Assume $E (X_i) = 0$, (2.A)---(2.B), and (\ref{eq:form of thetaip}) holds with
\begin{eqnarray}\label{eq:chi0}
0<\chi<\chi_0=\frac{p^2-4+(p-2)\sqrt{p^2+20p+4}}{8p}, \\ \label{eq:A condition}
A>\frac{(2p+p^2)\chi+p^2+3p+2+f_5^{1/2}}{p(1+p+2\chi)}.
\end{eqnarray} 
Then, (\ref{eq:Main Gaussian approximation eq}) holds with the approximation bound $\tau_n = n^{1/r}$, where
\begin{eqnarray}\label{eq:taun}
\frac{1}{r}= \frac{f_1+p^2 \chi+p^2-2p+f_2- \chi \sqrt{(p-2) (f_3-3p)}}{f_4}. 
\end{eqnarray}
\end{theorem}

\begin{theorem}\label{th:theorem 2}
Assume $E (X_i) = 0$, (2.A)---(2.B), and (\ref{eq:form of thetaip}) hold. Recall (\ref{eq:chi0}) for $\chi_0$: (i) if $\chi>\chi_0$ and $A>0$, we can achieve (\ref{eq:Main Gaussian approximation eq}) with $\tau_n = n^{1/p}$ for all $p>2$; for $\chi=\chi_0$, assume that $A$ satisfies (\ref{eq:A condition}); (ii) if $2<p<4$, we have $\tau_n = n^{1/p}$; (iii) if $p \geq 4$, we have $\tau_n = n^{1/p}\log n$.
\end{theorem}

Theorems \ref{th:main theorem} and \ref{th:theorem 2} concern the two cases $\chi < \chi_0$ and $\chi \ge \chi_0$, respectively, and they are proved in sections \ref{sec:proof1} and \ref{sec:poc} respectively. The proof of Theorem \ref{th:theorem 2} requires a more refined treatment so that the optimal rate can be derived. For Theorem \ref{th:main theorem} and Theorem \ref{th:theorem 2}(i) and (iii), we apply G{\"o}tze and Zaitsev (2008, \cite{MR2760567}); see Proposition $\ref{re:zaitsev result}$. For Theorem \ref{th:theorem 2}(ii), Proposition 1 from Einmahl (1987, \cite{MR899446}) is applied. The expression of $r$ is complicated. Figure \ref{fig:tau} plots the power $\max(1/r, 1/p)$. As $\chi \to 0$, $r \to 2$ and $r = p$ if $\chi >\chi_0$.

\begin{figure} [h!]  
\centering
\includegraphics[width=8.2cm,keepaspectratio]{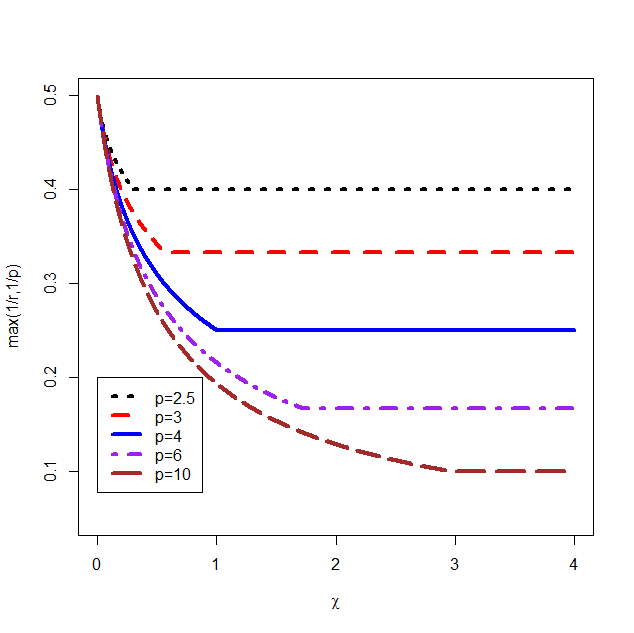} 
\caption{Optimal bound as a function of $\chi$}
\label{fig:tau} 
\end{figure}

\begin{remark} 
The lower bound of $A$ for the case $\chi=\chi_0$ can be further simplified to $$A > \frac{p^2+8p+4+(p-2)\sqrt{p^2+20p+4}}{6p}.$$  
\end{remark}

\section{Applications}
\label{sec:appl}
\subsection{Vector linear processes:}
\ignore{
Consider a vector ARMA process with lags $p$ and $q$
\begin{eqnarray}
X_i-A_1X_{i-1}-\ldots -A_pX_{i-p}=\epsilon_i+B_1\epsilon_{i-1}+\ldots+B_q\epsilon_{i-q}.
\end{eqnarray}
Because of the stationary nature of the process condition (2.A) is automatically satisfied. For condition (2.B) one needs invertibility of the process and the non-singularity of $A^{-1}B\Sigma B \tran A^{-\mkern-1.5mu\mathsf{T}}$ where $\Sigma$ is the dispersion matrix of the process $(\epsilon_i)_{i \ge 0}$ and 

\begin{eqnarray}\label{eq:A and B}
A
\end{eqnarray}
}

Assume that $X_i$ is a vector linear process
\begin{equation}\label{eq:veclin}
    X_i = \sum_{j=0}^\infty B_j \epsilon_{i-j},
\end{equation}
where $B_j$ is $d \times d$ coefficient matrix, and $\epsilon_i = (\epsilon_{i1}, \ldots, \epsilon_{i d})\tran$. Here $\epsilon_{i}$ is an i.i.d. random variable with mean zero and a finite $q$th moment, for some $q>2$. Assume
\begin{equation}\label{eq:condition on Bj}
    \sum_{j=t}^\infty |B_j| = O (t^{-\chi}(\log t)^{-A}),
\end{equation}
where $A$ satisfies (\ref{eq:A condition}), with $p$ therein replaced by $q$. The model in (\ref{eq:veclin}) covers a large class of popular multiple timeseries models including the vector AR, vector MA and vector ARMA models. under mild conditions on the coefficient matrices. Specifically, for a zero-mean vector ARMA process with lags $a$ and $b$ 
\begin{eqnarray}
X_i-\Psi_1X_{i-1}-\ldots -\Psi_aX_{i-a}=\epsilon_i+\Phi_1\epsilon_{i-1}+\ldots+\Phi_b\epsilon_{i-b},
\end{eqnarray}

\noindent the stability condition (see \cite{MR2172368} for a definition) ensures a pure vector MA representation (\ref{eq:veclin}). The stationarity of the $X_i$ process and the finite $q$th moment ensure condition (2.A), with $p$ replaced by $q$. Write $\Psi_*=I-\Psi_1-\ldots-\Psi_a, \Phi_*=I+\Phi_1+\ldots+\Phi_b.$ Assume $\Psi_*$, $\Phi_*$, and $\Sigma_e=E(e_1e_1\tran)$ are nonsingular. Elementary calculation shows that, as $l \rightarrow \infty$, $$Var(S_l/\sqrt{l}) \rightarrow \Psi_*^{-1}\Phi_* \Sigma_e \Phi_* \tran \Psi_*^{-\mkern-1.5mu\mathsf{T}},$$
which is also non-singular. Thus condition (2.B) holds. Note that $\|X_i-X_{i,(i-j)} \|_q=O(|B_j|)$. Therefore, condition (2.3) is satisfied for the $X_i$ process, from assumption (\ref{eq:condition on Bj}). Thus, under a suitable moment assumption, we can apply Theorems \ref{th:main theorem} and \ref{th:theorem 2} to generalize the central limit theory-type results to a stronger invariance principle.

Next, we discuss the covariance process for $X_i$ that admits a representation as (\ref{eq:veclin}). Assume $q > 4$. Let the $d(d+1)/2$-dimensional vector $W_i = (X_{i r} X_{i s})_{1 \le r \le s \le d}$. Then, $\bar W_n := \sum_{i=1}^n W_i / n$ gives sample covariances of $(X_i)_{i=1}^n$.  Write $p = q/2$. Fix two coordinates $1\leq r\leq s\leq d$. Then,
\begin{eqnarray*}
&&\|X_{ir}X_{is}-X_{i,(i-j)r}X_{i,(i-j)s}\|_p \nonumber \\
&& \quad \leq \|X_{ir}X_{is}- X_{ir}X_{i,(i-j)s}\|_p+ \|X_{ir}X_{i,(i-j)s}-X_{i,(i-j)r}X_{i,(i-j)s}\|_p \\
&& \quad \leq\|X_{ir}\|_q\|X_{is}-X_{i,(i-j)s}\|_q+\|X_{ir}-X_{i,(i-j)r}\|_q\|X_{i,(i-j)s}\|_q \\
&& \quad = O(|B_j|),
\end{eqnarray*}
because $\epsilon_i$ has a finite $q$th moment. Thus, condition (\ref{eq:condition on Bj}) translates to condition (\ref{eq:form of thetaip}) for the $W$ process with $p=q/2$. Condition (2.A) is trivially satisfied because the process $W_i$ is stationary and has a finite $p$th moment. Let $\Sigma_W = \sum_{k=-\infty}^{\infty} Cov(W_0, W_k)$ be the long-run covariance matrix of $(W_i)$. We assume the minimum eigenvalue of $\Sigma_W$ is positive. This ensures that condition (2.B) holds. By Theorems \ref{th:main theorem} and \ref{th:theorem 2}, we have 
\begin{equation}\label{eq:estimation of cov} 
    \max_{i\le n} |i \bar W_i - i E(W_1) - \Sigma_W^{1/2} I\!B(i)| = o_P(\tau_n),
\end{equation}
where $\tau_n$ takes the values $n^{1/r}$ (see (\ref{eq:taun})), and $n^{1/p}$, based on $\chi<\chi_0$ and $\chi >\chi_0$, respectively and $I\!B$ is a centered standard Brownian motion. Result (\ref{eq:estimation of cov}) is helpful for change point inferences for multiple time series based on covariances; see \cite{aue2009break, Trapani17}, among others.

\subsection{Nonlinear nonstationary time series:}  
Consider the process
\begin{eqnarray*}
X_{i} = F(X_{i-1}, \epsilon_i,\theta(i/n)), \,\, 1 \le i \le n,
\end{eqnarray*}
where $\epsilon_i$ is an i.i.d. random variable, $F$ is a measurable function, $\theta:[0,1] \to \mathbb{R}$ is a parametric function such that $\max_{0 \le u \le 1} \| F(x_0, \epsilon_i, \theta(u)) \|_p < \infty$, and
\begin{eqnarray}\label{eq:gmc2}
\sup_{0 \le u \le 1} \sup_{x \ne x'} \frac{ \| F(x, \epsilon_i,\theta(u))-F_i(x', \epsilon_i, \theta(u)) \|_p}{|x-x'|} < 1.
\end{eqnarray}
Then, the process $X_i$ satisfies the following geometric moment contraction: for some $0<\beta<1$, 
\begin{eqnarray}\label{eq:gmc}
\delta_{i,p} = O(\beta^{i}). 
\end{eqnarray}
Thus, (\ref{eq:form of thetaip}) holds for any $\chi > 0$, and Theorem \ref{th:theorem 2} is applicable with rate $\tau_n = n^{1/p}$. This facilitates an inference for the unknown parametric function $\theta$. Time-varying analogues of ARCH-, GARCH-, AR-, ARMA-type models are prominent examples in this large class of nonstationary models. We discuss the following example of a threshold AR(1) model (see Tong (1990, \cite{tong90})) with time-varying coefficients:
\begin{eqnarray}\label{eq:tar model}
Y_i=\theta_1(i/n)Y_{i-1}^++\theta_2(i/n)Y_{i-1}^-+e_i,
\end{eqnarray}
where $e_i$ is an i.i.d. mean-zero innovation. Assuming 
$\theta(\cdot)=(\theta_1(\cdot),\theta_2(\cdot)) \tran$ is continuous,  we can estimate $\theta(t)$, for $t \in [0,1]$, by 
\begin{eqnarray}\label{eq:estimate theta}
(\hat{\theta}_1(t),\hat{\theta}_2(t))^T=\arg \min_{\eta_1,\eta_2}\sum_{i=2}^{n}(Y_i-\eta_1Y_{i-1}^+-\eta_2Y_{i-1}^-)^2K\left(\frac{i/n-t}{b_n}\right),
\end{eqnarray}
where $K$ is a symmetric kernel with bounded variation and compact support, and $b_n$ is an appropriately chosen bandwidth. For such an estimation choice one has
\begin{eqnarray}\label{eq:imp eq}
\sqrt{nb_n}M(t)(\hat{\theta}(t)-\theta(t))&=&  \frac{1}{\sqrt{nb_n}}\sum_{i=2}^n \textbf{v}_i\textbf{v}_i\tran\left(\theta\left(\frac{i}{n}\right)-\theta(t)\right) K\left(\frac{i/n-t}{b_n}\right) \nonumber \\
&& \quad \quad \quad +\frac{1}{\sqrt{nb_n}}\sum_{i=2}^n \textbf{v}_ie_i K\left(\frac{i/n-t}{b_n}\right),
\end{eqnarray}
where $\textbf{v}_i=(Y_{i-1}^+, Y_{i-1}^- )\tran$ and $M(t)=(nb_n)^{-1}\sum_{i=2}^n \textbf{v}_i\textbf{v}_i\tran K ((i/n-t)/b_n)$. Assuming some mild conditions on the innovation process $e_i$ and the time-varying functions $\theta_1$ and $\theta_2$, we can construct a simultaneous confidence interval for $\theta$ from  (\ref{eq:imp eq}). Assume for some $p>2, \|e_1\|_p<\infty,$ $e_1$ has a density with support $(-\infty,\infty)$, and 
\begin{eqnarray}\label{eq:unifbound}
s=\sup_t(|\theta_1(t)|+|\theta_2(t)|)<1.
\end{eqnarray} We verify the conditions of Theorem \ref{th:theorem 2} using the bivariate process $X_i=\textbf{v}_ie_i$. To prove (2.A), it suffices to show uniform integrability for $(|Y_{i}|^p)_{i \geq 1}$ for the model (\ref{eq:tar model}). It easily follows because $e_i$ is an i.i.d. innovation process with a finite $p$th moment, and
$$|Y_i| \leq |e_i|+s|Y_{i-1}| \leq \sum_{j=0}^{\infty}s^j |e_{i-j}|.$$
Thus, (2.A) holds. As a result of the independence of $e_i$, and beacuse $x^+x^-=0$,
$$ Var(S_{t+l}-S_t)=\sum_{i=t+1}^{t+l} Var(\textbf{v}_ie_i)=\sum_{i=t+1}^{t+l} \text{diag}(E((Y_{i-1}^+)^2)E(e_i^2), E((Y_{i-1}^-)^2)E(e_i^2)).$$
With $D_i=\theta_1(i/n)Y_{i-1}^++\theta_2(i/n)Y_{i-1}^-$ and $c_0=2\sup_i\|Y_i\|_2$,
\begin{eqnarray}
E((Y_{i-1}^+)^2)=E(((e_{i-1}+D_{i-2})^+)^2) &\ge& E(((e_{i-1}+D_{i-2})^+)^2I(|D_{i-2}|\leq c_0)) \nonumber \\
&\ge& E(((e_{i-1}-c_0)^+)^2) P(|D_{i-2}| \leq c_0) \nonumber \\ 
&>& c_1(1-2\sup_i\|Y_i\|_2^2/c_0^2),
\end{eqnarray}
where $c_1$ is a constant that does not depend on $i$. We have a similar calculation for $E((Y_{i-1}^-)^2)$, and thus, (2.B) is satisfied. Under  assumption (\ref{eq:unifbound}), because $X_i$ satisfies the geometric moment contraction property (\ref{eq:gmc2}), (2.3) holds for any $\chi>0$.

For the second term in (\ref{eq:imp eq}), we apply the Gaussian approximation from Theorem \ref{th:theorem 2} with rate $\tau_n=n^{1/p}$. Using summation-by-parts, the negligibility criterion for the term with the approximation rate requires 
\begin{eqnarray}\label{eq:cond on bn}
n^{1/p}/\sqrt{nb_n} \to 0,
\end{eqnarray}
assuming bounded variation of $K$ (cf., Zhao and Wu (2007,\cite{zhaowu07})). Now, assume $\theta_1(\cdot)$ and $\theta_2(\cdot)$ are H\"older-$\alpha$ continuous for some $\alpha<1/2$. For the negligibility of the first term in (\ref{eq:imp eq}) portraying we need $\sqrt{nb_n}b_n^{\alpha}\to 0$. This, along with (\ref{eq:cond on bn}) and $\alpha<1/2$, requires $p> 4$. This portrays one scenario among many that demands a sharper Gaussian approximation than $n^{1/4}$. One such is obtained in Theorem \ref{th:theorem 2}. In the regime of curve estimation, our result provides a strong tool by relaxing the smoothness assumption on the coefficient curves/functions. This example shows how to overcome the unavailability of a Taylor series expansion using the minimal H\"older-continuity property and a sharper Gaussian approximation.

\section{Key ideas of the proof of Theorem \ref{th:main theorem}}
\label{sec:proof1}
The proof of Theorem \ref{th:main theorem} is quite involved. Here, we provide a brief outline of the major components of the proof. In particular, we emphasize the difficulties that arise as a result of the nonstationarity and the vector-valued process, as well as the techniques we use to circumvent these problems. Because these techniques allow us to solve this problem in such a general manner, we believe it might be of interest to the reader to at least have an overview of the major steps. A detailed proof is provided in the online Supplementary Material.

The first part of our proof consists of a series of approximations to create almost independent blocks. The first of them, the truncation approximation, ensures the optimal $n^{1/p}$ bound. This step differs from the treatment of \cite{MR3178474} because of the choice of the truncation level; we included the term $t_n$, exploiting the uniform integrability assumption. This is necessary because of the nonstationarity. Second, we use the $m$-dependence approximation for a suitably chosen sequence $m_n$ in terms of the decay rate $\chi$. This generalizes the treatment in \cite{MR3178474} because it also allows for processes where dependence decays slowly. Lastly, the blocking approximation requires some sharp Rosenthal-type inequality that needs a $\gamma$th moment of the block-sums in the numerator with $\gamma>p$. It is essential to use a power higher than $p$ to obtain a better rate.  This step needs a $k$-dic decomposition, where $k$ is possibly greater than or equal to three, to allow for nonstationarity.

To maintain clarity, we defer the exact choice of $\gamma$ and $m_n$ in terms of $\chi$ and $A$ to subsection \ref{sec:concc}. Instead, in this subsection, we derive conditions (\ref{eq:conditionsshort}) (see (\ref{eq:first condition}), (\ref{eq:second condition}), and (\ref{eq:third condition}) in the online supplement A) to ensure an $n^{1/r}$ rate and to solve $\gamma,m_n$, and $r$ later to obtain the best possible choices for this sequence. Henceforth, we drop the suffix of $m_n$ for convenience.

\subsection{Outline of preparation step:} The importance of the preparation step is two-fold. It creates a platform for the conditional Gaussian approximation and regrouping by creating almost independent blocks. Moreover, these steps allow us to build a system of equations to solve for the approximation rate $\tau_n=n^{1/r}$ as a function of the decay rate $\chi$ in (\ref{eq:form of thetaip}). These equations are key in our generic approach deriving the optimal rate for slowly decaying dependence, and show how it possibly affects (see Figure \ref{fig:tau}) the optimal Gaussian approximation rate. 

For the truncating approximation, we exploit the uniform integrability to introduce a sequence $t_n \rightarrow$ 0 very slowly, such as  
\begin{eqnarray} \label{eq:tn part3}
t_n \log \log n \to \infty,
\end{eqnarray}
and use it at the truncation level $t_nn^{1/p}$. The truncation is defined through the operator  

$$T_b(v)=(T_b(v_1),\ldots, T_b(v_d))\tran, \mbox{ where } T_b(w) = \min (\max(w,-b),b).$$

\noindent For the $m$-dependence approximation step and the blocking approximation, assume 
\begin{eqnarray}\label{eq:define m}
 m = \lfloor n^Lt_n^k \rfloor, \quad 0<k<(\gamma-p)/(\gamma/2-1), \quad 0<L<1,
\end{eqnarray} 
\begin{eqnarray}\label{eq:conditionsshort}
n^{1/2-1/r} \Theta_{m,r} &\to& 0, \quad 
n^{1-\gamma/r} m^{\gamma/2-1} \to 0 \quad \text{and } \quad
n^{1/p-1/\gamma} \sum_{j=m+1}^{\infty} \delta_{j,p}^{p/\gamma} \to 0,
\end{eqnarray}
where the first term in (\ref{eq:conditionsshort}) is required for the $m$-dependence step, and the other two are for the blocking approximation. After these approximations, we have a partial sum process $S_n^{\diamond}$, with the following summarized definition:
\begin{eqnarray*}
S_i^{\diamond}&=& \sum_{j=1}^{q_i} A_j\quad \mbox{ with }\quad A_{j}=\sum_{i=(2jk_0-2k_0)m+1}^{2k_0jm} \tilde{X}_i,\\ 
\mbox{ where } \tilde{X}_{j}&=&E(T_{t_nn^{1/p}}(X_j)|\epsilon_{j}, \ldots, \epsilon_{j-m})- E(T_{t_nn^{1/p}}(X_j)), 
\end{eqnarray*}

\noindent and $k_0= \lfloor \Theta_{0,2}^2/\lambda_*\rfloor+2, q_i=  \lfloor i/(2k_0m) \rfloor$. For this truncated, $m$-dependent and blocked process $S_n^{\diamond}$, we have the approximation
\begin{eqnarray*}
\displaystyle\max_{1 \leq i \leq n} |S_i- S_i^{\diamond}|= o_P(n^{1/r}).
\end{eqnarray*}

\noindent See section \ref{ssc:prep stage} in the online Supplementary Material. Next, in subsections \ref{sec:caga} and \ref{sec:rcr}, we discuss how to obtain a Gaussian approximation for $S_n^{\diamond}$.

\ignore{
Thus these steps although presented in the same sequential order as of Berkes, Liu and Wu (2014, \cite{MR3178474}) might look like , these are essentially new steps that were necessary to accommodate a much larger non-stationary class of vector-valued processes. 
}

\subsection{Outline of conditional Gaussian approximation:} 
\label{sec:caga}
The blocks created in the preparation steps are not independent because two successive blocks share some $\epsilon_i$ in their shared border. In this second stage, we consider the partial sum process conditioned on these borderline $\epsilon_i$, which implies conditional independence. Berkes, Liu, and Wu (2014, \cite{MR3178474}) performed a similar treatment with a triadic decomposition for stationary scalar processes, and applied Sakhanenko's (2006, \cite{MR2302850}) Gaussian approximation result to the conditioned process.

Because the result of Sakhanenko (2006, \cite{MR2302850}) is only valid for $d=1$, we need to use the Gaussian approximation result from G{\"o}tze and Zaitsev (2008, \cite{MR2760567}) (see Proposition \ref{re:zaitsev result}) for $d \ge 2$. This incurs a cost of verifying a very technical sufficient condition on the covariance matrices of the independent vectors. This verification is particularly complicated in our case because we are dealing with a conditional process. We opt for a $k$-dic decomposition instead of the triadic decomposition in \cite{MR3178474}. This is necessary to accommodate the nonstationarity of the process.  We need $k_0> \Theta_{0,2}^2/\lambda_*$ (cf., (\ref{eq:define big block})), where $\lambda_*$ is mentioned in Condition 2.B.

\subsection{Outline of regrouping and unconditional Gaussian approximation:} 
\label{sec:rcr}
In the last part of our proof, we obtain the Gaussian approximation for the unconditional process by applying Proposition \ref{re:zaitsev result} one more time. In the second part of our proof, we consider the conditional variance (cf., $V_j(\bar{a}_{2k_0j},\bar{a}_{2k_0j+2k_0})=Var(Y_j(\bar{a}_{2k_0j},\bar{a}_{2k_0j+2k_0}))$ in (\ref{eq:Demeaned process}) of subsection \ref{sec:cga}) of the blocks. These conditional variances are one-dependent. In order to apply G{\"o}tze and Zaitsev's (2008, \cite{MR2760567}) result, we rearrange the sums of these variances into sums of independent blocks (cf., \ref{eq:identity of variance} in subsection \ref{sec:cga}). Owing to the nonstationarity, this regrouping is different and more complex than that of Berkes, Liu, and Wu (2014, \cite{MR3178474}). In particular, the regrouping procedure leads to matrices that may not be positive-definite and, hence, cannot be used directly as possible covariance matrices of Gaussian processes. We overcome this obstacle by introducing a novel positive-definitization that does not affect the optimal rate.

\subsection{Conclusion of the proof:} 
\label{sec:concc}
This subsection discusses the choice of the sequence $m, \gamma$, and the rate $\tau_n=n^{1/r}$, starting from the conditions in (\ref{eq:conditionsshort}) (see equations (\ref{eq:first condition}), (\ref{eq:second condition}), and (\ref{eq:third condition}) in the detailed version of the proof). Elementary calculations show that $r < p$ for $\chi<\chi_0$. Provided $1- (\chi+1)p/\gamma <0$, we have
\begin{eqnarray}\label{eq:third condition simplify}
\sum_{j=m+1}^{\infty} \delta_{j,p}^{p/\gamma} &\leq& \sum_{i=\lfloor\log_2 m\rfloor}^{\infty}\sum_{j=2^i}^{2^{i+1}-1} \delta_{j,p}^{p/\gamma} 
\leq \sum_{i=\lfloor\log_2 m\rfloor}^{\infty} 2^{i (1- p/\gamma)} \Theta_{2^i,p}^{p/\gamma} \\
\nonumber & =& \sum_{i=\lfloor\log_2 m\rfloor}^{\infty} 2^{i (1- p/\gamma)} O(2^{-\chi i p /\gamma} i^{-Ap/\gamma}) = O(m^{1- p/\gamma- \chi p /\gamma} (\log m)^{-Ap/\gamma}).
\end{eqnarray}

\noindent By (\ref{eq:tn part3}) and (\ref{eq:define m dup}), $\log m \asymp \log n.$ Assume that
\begin{eqnarray}
\label{eq:equation 1}
1/2-1/r - \chi L &=& 0,\quad A>\gamma/p, \\
\label{eq:equation 2}
1-\gamma/r+L(\gamma/2-1)&=&0, \quad  0<k<(\gamma/2-1)^{-1}(\gamma-p)\\\
\label{eq:equation 3}
1/p-1/\gamma + (1- (\chi+1)p/\gamma)L &=&  0.
\end{eqnarray}
Then, the conditions in (\ref{eq:conditionsshort}) hold. Solving the equations in (\ref{eq:equation 1}), (\ref{eq:equation 2}), and (\ref{eq:equation 3}), we obtain $r$ in (\ref{eq:taun}), as follows:
\begin{eqnarray*}
\gamma &=&\frac{(2p+p^2)\chi+p^2+3p+2+f_5^{1/2}}{2+2p+4\chi},\\
L &=& \frac{ f_1-f_2+\chi\sqrt{(p-2) (f_3-3p)}}{\chi f_4},
\end{eqnarray*}
with $f_1, \ldots, f_5$ given in (\ref{eq:fs}). Moreover, we specifically choose $A>2\gamma/p$ for a crucial step in the proof of our Gaussian approximation; see (\ref{eq:first term part 2}).

\begin{remark}
Figure \ref{fig:gamma and L} depicts how $\gamma$ and $L$ change with $p$ and $\chi$ for $\chi<\chi_0$. Note that $L$, the power of $n$ in the expression of $m$, is close to one if $\chi$ is small. This makes intuitive sense, because if the dependence decays very slowly, to make blocks of size $m$ (or a multiple of $m$) behave almost independently, we need a larger $L$.
\end{remark} 

\begin{figure} [h!]   
\centering
\begin{tabular}{cc}
\subfigure[]{\epsfig{file=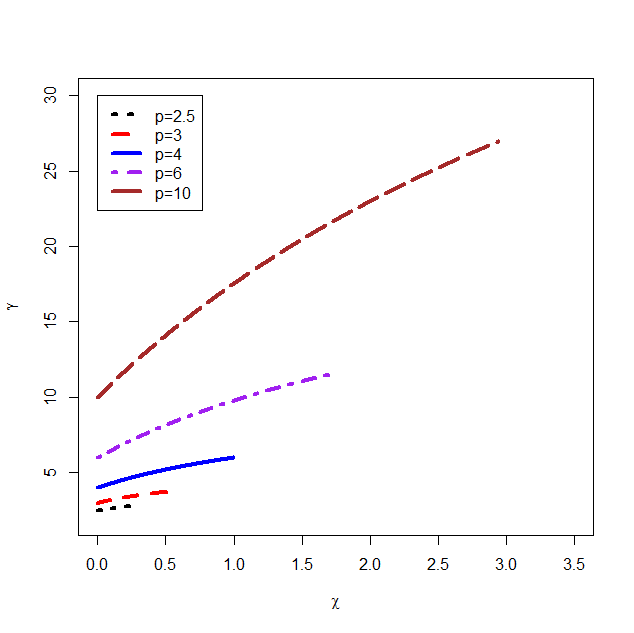, width=7 cm}}& 
\subfigure[]{\epsfig{file=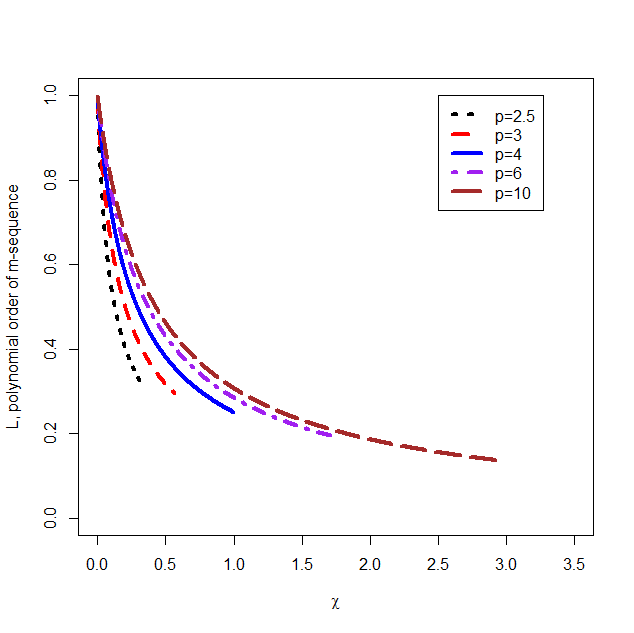, width=7 cm}}\\
\end{tabular}
\caption{(a) $\gamma$ as a function of $\chi$, (b) $L$ as a function of $\chi$}
\label{fig:gamma and L}
\end{figure}

\section{Proof of Theorem \ref{th:theorem 2}}
\label{sec:poc}
\begin{proof} \textit{Case 1 ($\chi>\chi_0$):} Note that the optimal power $\gamma$ and the optimal bound $1/r$ increase and decrease with $\chi$, respectively (see also Figures \ref{fig:tau} and \ref{fig:gamma and L}). This is a motivation behind tweaking our proof for the verification of (\ref{eq:zaitsev condition 1}) to handle the $(\log n)$ term in the choice of $l$ in (\ref{eq:choice of l}). When using the Nagaev inequality to show (\ref{eq:vja-evja small 2}), we use a power $\gamma'>\gamma$, while keeping the choice of $l$ (cf., \ref{eq:choice of l}) the same as before. We form a set of new equations:
\begin{eqnarray}\label{eq:new set}
1/2+1/p-2/r'+L'(1-(\chi+1)p/r')&=&0, \\  \nonumber
1/p-1/\gamma'+L'-L'(\chi+1)p/\gamma'&=&0, \\ \nonumber
1-\gamma'/r'+L'(\gamma'/2-1)&=&0.
\end{eqnarray}
The intuition behind the first of these equations is to use a higher power than $p$ in the $m$-dependence approximation. However, we have only defined moments up to $p$. Therefore, we use Lemma \ref{lem:truncated deltajy} to obtain a new equation corresponding to the $m$-dependence approximation using a power $r'$ that is little higher than $p$. The solution of (\ref{eq:new set}) has the property
\begin{eqnarray}\label{eq:imp property}
\gamma'<2(1+p+p\chi)/3,
\end{eqnarray} \noindent for $\chi>\chi_0$. In addition, $L'<L(\chi_0)$ (cf., Figure \ref{fig:gamma and L}) and, hence, $m^{1-\gamma'/2} \ll m'^{1-\gamma'/2}$, where $m'$ is taken as $n^{L'}t_n^{k}$, following (\ref{eq:define m dup}).  We apply Nagaev-type inequality from Liu, Xiao, and Wu  (2013, \cite{MR3114713}) to obtain
\begin{eqnarray}\label{eq:nagaev gamma}
 P(|\tilde{S}_m| \ge \sqrt{lm}) &\lesssim& \frac{m}{(lm)^{\gamma'/2}}\nu_R^{\gamma'+1}+\sum_{r=1}^R \exp \left(-c_{\gamma'} \frac{\lambda_r^2 l}{\tilde{\theta}_{r,2}^2}\right) + \frac{m^{\gamma'/2}\tilde{\Theta}_{m+1,\gamma'}^{\gamma'}}{(lm)^{\gamma'/2}}\\ 
\nonumber &+&\frac{m \sup_i\|T_{t_nn^{1/p}}(X_i)\|_{\gamma'}^{\gamma'}}{(lm)^{\gamma'/2}}+ \exp \left(-\frac{c_{\gamma'}l}{\sup_i\|T_{t_nn^{1/p}}(X_i)\|_2^2}\right),
\end{eqnarray}
where $\nu_R=\sum_{r=1}^{R}\mu_r$, $\mu_r= (\tau_r^{\gamma'/2-1} \tilde{\theta}_{r,\gamma'}^{\gamma'})^{1/(\gamma'+1)}$, $\lambda_r =\mu_r/\nu_R$, and  $\tilde{\theta}_{r,t}=\sum_{i=1+\tau_{r-1}}^{\tau_r}\tilde{\delta}_{i,t}$, for some sequence $0=\tau_0 <\tau_1< \ldots <\tau_R=m$. For the choice $\tau_r=2^{r-1}$ for $1 \leq r \leq R-1=\lfloor \log_2 m \rfloor$, we obtain $\nu_R^{\gamma'+1}=O(n^{\gamma'/p-1}t_n^{\gamma'-p})$ using (\ref{eq:imp property}), or (\ref{eq:tn part1}) under the decay condition on $\Theta_{i,p}$ in (\ref{eq:form of thetaip}). The third term and the exponential terms are straightforward to deal with. The fourth term is handled similarly to (\ref{eq:third term}). Combining these as in our new set of equations in (\ref{eq:new set}), we get $P(|\tilde{S}_m|\ge \sqrt{lm})=o(m/n)$, which is sufficient to conclude the proof, as proposed in (\ref{eq:vja-evja small 2}).

The positive-definitization technique introduced in (\ref{eq:define Vj1}) is validated in Proposition \ref{prop:Vj final}. This step requires $\gamma>4\chi$ for $\chi>\max(1/2,\chi_0)$. We observe that $\gamma'-4\chi=0$ has a root  $\chi_1>\chi_0$. This allows us to replace $\chi$ in the decay condition of $\Theta_{i,p}$ with $\min(\chi,\chi_1)$, and thus completes the proof. The arguments for the rest of the proof of Theorem \ref{th:main theorem} remain valid.

\textit{Case 2 ($\chi=\chi_0, 2 < p < 4$):} We apply Proposition 1 from Einmahl (1987, \cite{MR899446}). He proved a Gaussian approximation result for independent, but not necessarily identical vectors with a diagonal covariance matrix. The two remarks following the proposition mention that the diagonal nature of every covariance matrix can be relaxed if these matrices have bounded eigenvalues. A careful check of his proof reveals that it can be further relaxed to the assumption of bounded eigenvalues of the covariance matrix of a normalized block sum only. This allows us to replace $l$ (see (\ref{eq:choice of l})) in the conclusion of Proposition \ref{re:zaitsev result} with $l'$ without the logarithm term $(\log n)$ in the denominator and without the condition (\ref{eq:zaitsev condition 2}). Thus, we obtain a rate of $o_P(n^{1/p})$ for all $2<p<4$.

\textit{Case 3 ($\chi=\chi_0, p \ge 4$):} In this case, we do not have a similar optimal Gaussian approximation result for independent, but not identically distributed random vectors. Instead we apply Proposition \ref{re:zaitsev result} again. The sufficient conditions in that result lead to an unavoidable $(\log n)$ term in the choice of $l$ (see \ref{eq:choice of l}). This, in turn, leads to a rate of $o_P(n^{1/p} \log n) $. Note that $\chi_0>1/2-1/p$ for all $p>2$. From the proof of the case $0<\chi<\chi_0$, consider (\ref{eq:nagaev difficult case}). Then, observe that if $\chi=\chi_0$, 
\begin{equation*}
    \frac{n}{m}P(|\tilde{S}_m| \ge \sqrt{lm}) =O((\log n)^p t_n^{k(p/\gamma- p/2)}),
\end{equation*}
which may diverge to $\infty$. To deal with this difficulty in this special case,  we choose a different $m$ sequence. Our new set of conditions with $\tau_n=n^{1/p} (\log n)^{\delta}$ are
\begin{eqnarray*}
n^{1/2-1/p}m^{-\chi} (\log n)^{-A-\delta} &\to& 0, \\
n^{1/p-1/\gamma} m^{1-(\chi+1)p/\gamma} (\log n)^{-Ap/\gamma} &\to& 0, \\
n^{1-\gamma/p} (\log n)^{-\gamma \delta} m^{\gamma/2-1} &\to& 0, \\
(\log n)^{\gamma} m^{1-\gamma/2} n^{\gamma/p-1}t_n^{\gamma-p} &\to& 0, 
\end{eqnarray*}
\noindent where the last is obtained using $\gamma$th moment in (\ref{eq:nagaev gamma}). Let $m=\lfloor n^L (\log n)^{2\gamma/(\gamma-2)}t_n^{k} \rfloor$, with $0<k<(\gamma/2-1)^{-1}(\gamma-p)$. Then, we can achieve $\delta=1$. We still have the same set of equations for $L,\gamma$, and $r$ shown in (\ref{eq:equation 1}), (\ref{eq:equation 2}), and (\ref{eq:equation 3}), respectively. A careful check reveals that the rest of the proof follows with this modified $m$ sequence. 
\end{proof}

\noindent \textbf{Supplementary Material}  

The online Supplementary Material contains detailed proofs of Theorem \ref{th:main theorem} (section \ref{sec:proofdetails}) and some useful lemmas (section \ref{sec:sul}).

\noindent \textbf{Acknowledgements}

We are grateful to the Associate Editor and an anonymous referee for their helpful feedback and comments. This study was partially supported by NSF/DMS 1405410.

\bibliographystyle{imsart-nameyear}
\bibliography{gaussian}

\begin{thebibliography}{35}

\bibitem[\protect\citeauthoryear{Aue et~al.}{2009}]{aue2009break}
\begin{barticle}[author]
\bauthor{\bsnm{Aue},~\bfnm{Alexander}\binits{A.}},
  \bauthor{\bsnm{H{\"o}rmann},~\bfnm{Siegfried}\binits{S.}},
  \bauthor{\bsnm{Horv{\'a}th},~\bfnm{Lajos}\binits{L.}} \AND
  \bauthor{\bsnm{Reimherr},~\bfnm{Matthew}\binits{M.}}
(\byear{2009}).
\btitle{Break detection in the covariance structure of multivariate time series
  models}.
\bjournal{The Annals of Statistics}
\bvolume{37}
\bpages{4046--4087}.
\end{barticle}
\endbibitem

\bibitem[\protect\citeauthoryear{Berkes, Liu and Wu}{2014}]{MR3178474}
\begin{barticle}[author]
\bauthor{\bsnm{Berkes},~\bfnm{Istv{\'a}n}\binits{I.}},
  \bauthor{\bsnm{Liu},~\bfnm{Weidong}\binits{W.}} \AND
  \bauthor{\bsnm{Wu},~\bfnm{Wei~Biao}\binits{W.~B.}}
(\byear{2014}).
\btitle{Koml\'os-{M}ajor-{T}usn\'ady approximation under dependence}.
\bjournal{Ann. Probab.}
\bvolume{42}
\bpages{794--817}.
\bdoi{10.1214/13-AOP850}
\bmrnumber{3178474}
\end{barticle}
\endbibitem

\bibitem[\protect\citeauthoryear{Donsker}{1952}]{MR0047288}
\begin{barticle}[author]
\bauthor{\bsnm{Donsker},~\bfnm{Monroe~D.}\binits{M.~D.}}
(\byear{1952}).
\btitle{Justification and extension of {D}oob's heuristic approach to the
  {K}omogorov-{S}mirnov theorems}.
\bjournal{Ann. Math. Statistics}
\bvolume{23}
\bpages{277--281}.
\bmrnumber{0047288}
\end{barticle}
\endbibitem

\bibitem[\protect\citeauthoryear{Doob}{1949}]{MR0030732}
\begin{barticle}[author]
\bauthor{\bsnm{Doob},~\bfnm{J.~L.}\binits{J.~L.}}
(\byear{1949}).
\btitle{Heuristic approach to the {K}olmogorov-{S}mirnov theorems}.
\bjournal{Ann. Math. Statistics}
\bvolume{20}
\bpages{393--403}.
\bmrnumber{0030732}
\end{barticle}
\endbibitem

\bibitem[\protect\citeauthoryear{Eberlein}{1986}]{MR815969}
\begin{barticle}[author]
\bauthor{\bsnm{Eberlein},~\bfnm{Ernst}\binits{E.}}
(\byear{1986}).
\btitle{On strong invariance principles under dependence assumptions}.
\bjournal{Ann. Probab.}
\bvolume{14}
\bpages{260--270}.
\bmrnumber{815969}
\end{barticle}
\endbibitem

\bibitem[\protect\citeauthoryear{Einmahl}{1987a}]{MR899446}
\begin{barticle}[author]
\bauthor{\bsnm{Einmahl},~\bfnm{Uwe}\binits{U.}}
(\byear{1987}a).
\btitle{A useful estimate in the multidimensional invariance principle}.
\bjournal{Probab. Theory Related Fields}
\bvolume{76}
\bpages{81--101}.
\bdoi{10.1007/BF00390277}
\bmrnumber{899446}
\end{barticle}
\endbibitem

\bibitem[\protect\citeauthoryear{Einmahl}{1987b}]{MR905340}
\begin{barticle}[author]
\bauthor{\bsnm{Einmahl},~\bfnm{Uwe}\binits{U.}}
(\byear{1987}b).
\btitle{Strong invariance principles for partial sums of independent random
  vectors}.
\bjournal{Ann. Probab.}
\bvolume{15}
\bpages{1419--1440}.
\bmrnumber{905340}
\end{barticle}
\endbibitem

\bibitem[\protect\citeauthoryear{Einmahl}{1989}]{MR996984}
\begin{barticle}[author]
\bauthor{\bsnm{Einmahl},~\bfnm{Uwe}\binits{U.}}
(\byear{1989}).
\btitle{Extensions of results of {K}oml\'os, {M}ajor, and {T}usn\'ady to the
  multivariate case}.
\bjournal{J. Multivariate Anal.}
\bvolume{28}
\bpages{20--68}.
\bdoi{10.1016/0047-259X(89)90097-3}
\bmrnumber{996984}
\end{barticle}
\endbibitem

\bibitem[\protect\citeauthoryear{Erd{\"o}s and Kac}{1946}]{MR0015705}
\begin{barticle}[author]
\bauthor{\bsnm{Erd{\"o}s},~\bfnm{P.}\binits{P.}} \AND
  \bauthor{\bsnm{Kac},~\bfnm{M.}\binits{M.}}
(\byear{1946}).
\btitle{On certain limit theorems of the theory of probability}.
\bjournal{Bull. Amer. Math. Soc.}
\bvolume{52}
\bpages{292--302}.
\bmrnumber{0015705}
\end{barticle}
\endbibitem

\bibitem[\protect\citeauthoryear{G{\"o}tze and Zaitsev}{2008}]{MR2760567}
\begin{barticle}[author]
\bauthor{\bsnm{G{\"o}tze},~\bfnm{F.}\binits{F.}} \AND
  \bauthor{\bsnm{Zaitsev},~\bfnm{A.~Yu.}\binits{A.~Y.}}
(\byear{2008}).
\btitle{Bounds for the rate of strong approximation in the multidimensional
  invariance principle}.
\bjournal{Teor. Veroyatn. Primen.}
\bvolume{53}
\bpages{100--123}.
\bdoi{10.1137/S0040585X9798350X}
\bmrnumber{2760567}
\end{barticle}
\endbibitem

\bibitem[\protect\citeauthoryear{Koml{\'o}s, Major and
  Tusn{\'a}dy}{1975}]{MR0375412}
\begin{barticle}[author]
\bauthor{\bsnm{Koml{\'o}s},~\bfnm{J.}\binits{J.}},
  \bauthor{\bsnm{Major},~\bfnm{P.}\binits{P.}} \AND
  \bauthor{\bsnm{Tusn{\'a}dy},~\bfnm{G.}\binits{G.}}
(\byear{1975}).
\btitle{An approximation of partial sums of independent {${\rm RV}$}'s and the
  sample {${\rm DF}$}. {I}}.
\bjournal{Z. Wahrscheinlichkeitstheorie und Verw. Gebiete}
\bvolume{32}
\bpages{111--131}.
\bmrnumber{0375412}
\end{barticle}
\endbibitem

\bibitem[\protect\citeauthoryear{Koml{\'o}s, Major and
  Tusn{\'a}dy}{1976}]{MR0402883}
\begin{barticle}[author]
\bauthor{\bsnm{Koml{\'o}s},~\bfnm{J.}\binits{J.}},
  \bauthor{\bsnm{Major},~\bfnm{P.}\binits{P.}} \AND
  \bauthor{\bsnm{Tusn{\'a}dy},~\bfnm{G.}\binits{G.}}
(\byear{1976}).
\btitle{An approximation of partial sums of independent {RV}'s, and the sample
  {DF}. {II}}.
\bjournal{Z. Wahrscheinlichkeitstheorie und Verw. Gebiete}
\bvolume{34}
\bpages{33--58}.
\bmrnumber{0402883}
\end{barticle}
\endbibitem

\bibitem[\protect\citeauthoryear{Liu and Lin}{2009}]{MR2485027}
\begin{barticle}[author]
\bauthor{\bsnm{Liu},~\bfnm{Weidong}\binits{W.}} \AND
  \bauthor{\bsnm{Lin},~\bfnm{Zhengyan}\binits{Z.}}
(\byear{2009}).
\btitle{Strong approximation for a class of stationary processes}.
\bjournal{Stochastic Process. Appl.}
\bvolume{119}
\bpages{249--280}.
\bdoi{10.1016/j.spa.2008.01.012}
\bmrnumber{2485027}
\end{barticle}
\endbibitem

\bibitem[\protect\citeauthoryear{Liu and Wu}{2010}]{MR2660298}
\begin{barticle}[author]
\bauthor{\bsnm{Liu},~\bfnm{Weidong}\binits{W.}} \AND
  \bauthor{\bsnm{Wu},~\bfnm{Wei~Biao}\binits{W.~B.}}
(\byear{2010}).
\btitle{Asymptotics of spectral density estimates}.
\bjournal{Econometric Theory}
\bvolume{26}
\bpages{1218--1245}.
\bdoi{10.1017/S026646660999051X}
\bmrnumber{2660298}
\end{barticle}
\endbibitem

\bibitem[\protect\citeauthoryear{Liu, Xiao and Wu}{2013}]{MR3114713}
\begin{barticle}[author]
\bauthor{\bsnm{Liu},~\bfnm{Weidong}\binits{W.}},
  \bauthor{\bsnm{Xiao},~\bfnm{Han}\binits{H.}} \AND
  \bauthor{\bsnm{Wu},~\bfnm{Wei~Biao}\binits{W.~B.}}
(\byear{2013}).
\btitle{Probability and moment inequalities under dependence}.
\bjournal{Statist. Sinica}
\bvolume{23}
\bpages{1257--1272}.
\bmrnumber{3114713}
\end{barticle}
\endbibitem

\bibitem[\protect\citeauthoryear{L\"utkepohl}{2005}]{MR2172368}
\begin{bbook}[author]
\bauthor{\bsnm{L\"utkepohl},~\bfnm{Helmut}\binits{H.}}
(\byear{2005}).
\btitle{New introduction to multiple time series analysis}.
\bpublisher{Springer-Verlag, Berlin}.
\bdoi{10.1007/978-3-540-27752-1}
\bmrnumber{2172368}
\end{bbook}
\endbibitem

\bibitem[\protect\citeauthoryear{Monrad and Philipp}{1991}]{monrad1991problem}
\begin{barticle}[author]
\bauthor{\bsnm{Monrad},~\bfnm{Ditlev}\binits{D.}} \AND
  \bauthor{\bsnm{Philipp},~\bfnm{Walter}\binits{W.}}
(\byear{1991}).
\btitle{The problem of embedding vector-valued martingales in a Gaussian
  process}.
\bjournal{Theory of Probability \& Its Applications}
\bvolume{35}
\bpages{374--377}.
\end{barticle}
\endbibitem

\bibitem[\protect\citeauthoryear{Nagaev}{1979}]{MR542129}
\begin{barticle}[author]
\bauthor{\bsnm{Nagaev},~\bfnm{S.~V.}\binits{S.~V.}}
(\byear{1979}).
\btitle{Large deviations of sums of independent random variables}.
\bjournal{Ann. Probab.}
\bvolume{7}
\bpages{745--789}.
\bmrnumber{542129}
\end{barticle}
\endbibitem

\bibitem[\protect\citeauthoryear{Priestley}{1988}]{MR991969}
\begin{bbook}[author]
\bauthor{\bsnm{Priestley},~\bfnm{M.~B.}\binits{M.~B.}}
(\byear{1988}).
\btitle{Nonlinear and nonstationary time series analysis}.
\bpublisher{Academic Press, Inc. [Harcourt Brace Jovanovich, Publishers],
  London}.
\bmrnumber{991969}
\end{bbook}
\endbibitem

\bibitem[\protect\citeauthoryear{Prohorov}{1956}]{MR0084896}
\begin{barticle}[author]
\bauthor{\bsnm{Prohorov},~\bfnm{Yu.~V.}\binits{Y.~V.}}
(\byear{1956}).
\btitle{Convergence of random processes and limit theorems in probability
  theory}.
\bjournal{Teor. Veroyatnost. i Primenen.}
\bvolume{1}
\bpages{177--238}.
\bmrnumber{0084896}
\end{barticle}
\endbibitem

\bibitem[\protect\citeauthoryear{Sakhanenko}{2006}]{MR2302850}
\begin{barticle}[author]
\bauthor{\bsnm{Sakhanenko},~\bfnm{A.~I.}\binits{A.~I.}}
(\byear{2006}).
\btitle{Estimates in the invariance principle in terms of truncated power
  moments}.
\bjournal{Sibirsk. Mat. Zh.}
\bvolume{47}
\bpages{1355--1371}.
\bdoi{10.1007/s11202-006-0119-1}
\bmrnumber{2302850}
\end{barticle}
\endbibitem

\bibitem[\protect\citeauthoryear{Tong}{1990}]{tong90}
\begin{bbook}[author]
\bauthor{\bsnm{Tong},~\bfnm{Howell}\binits{H.}}
(\byear{1990}).
\btitle{Nonlinear time series}.
\bseries{Oxford Statistical Science Series}
\bvolume{6}.
\bpublisher{The Clarendon Press, Oxford University Press, New York}
\bnote{A dynamical system approach, With an appendix by K. S. Chan, Oxford
  Science Publications}.
\bmrnumber{1079320}
\end{bbook}
\endbibitem

\bibitem[\protect\citeauthoryear{Trapani, Urga and Kao}{2017}]{Trapani17}
\begin{barticle}[author]
\bauthor{\bsnm{Trapani},~\bfnm{L.}\binits{L.}},
  \bauthor{\bsnm{Urga},~\bfnm{G.}\binits{G.}} \AND
  \bauthor{\bsnm{Kao},~\bfnm{C}\binits{C.}}
(\byear{2017}).
\btitle{Testing for instability in covariance structures}.
\bjournal{Bernoulli}
\bpages{To Appear}.
\end{barticle}
\endbibitem

\bibitem[\protect\citeauthoryear{Tsay}{2010}]{MR2778591}
\begin{bbook}[author]
\bauthor{\bsnm{Tsay},~\bfnm{Ruey~S.}\binits{R.~S.}}
(\byear{2010}).
\btitle{Analysis of financial time series},
\bedition{third} ed.
\bseries{Wiley Series in Probability and Statistics}.
\bpublisher{John Wiley \& Sons, Inc., Hoboken, NJ}.
\bdoi{10.1002/9780470644560}
\bmrnumber{2778591}
\end{bbook}
\endbibitem

\bibitem[\protect\citeauthoryear{Wiener}{1958}]{wiener1958nonlinear}
\begin{bbook}[author]
\bauthor{\bsnm{Wiener},~\bfnm{N.}\binits{N.}}
(\byear{1958}).
\btitle{Nonlinear Problems in Random Theory}.
\bpublisher{Wiley, New York}.
\end{bbook}
\endbibitem

\bibitem[\protect\citeauthoryear{Wu}{2005}]{MR2172215}
\begin{barticle}[author]
\bauthor{\bsnm{Wu},~\bfnm{Wei~Biao}\binits{W.~B.}}
(\byear{2005}).
\btitle{Nonlinear system theory: another look at dependence}.
\bjournal{Proc. Natl. Acad. Sci. USA}
\bvolume{102}
\bpages{14150--14154 (electronic)}.
\bdoi{10.1073/pnas.0506715102}
\bmrnumber{2172215}
\end{barticle}
\endbibitem

\bibitem[\protect\citeauthoryear{Wu}{2007}]{MR2353389}
\begin{barticle}[author]
\bauthor{\bsnm{Wu},~\bfnm{Wei~Biao}\binits{W.~B.}}
(\byear{2007}).
\btitle{Strong invariance principles for dependent random variables}.
\bjournal{Ann. Probab.}
\bvolume{35}
\bpages{2294--2320}.
\bdoi{10.1214/009117907000000060}
\bmrnumber{2353389}
\end{barticle}
\endbibitem

\bibitem[\protect\citeauthoryear{Wu and Wu}{2016}]{MR3466186}
\begin{barticle}[author]
\bauthor{\bsnm{Wu},~\bfnm{Wei~Biao}\binits{W.~B.}} \AND
  \bauthor{\bsnm{Wu},~\bfnm{Ying~Nian}\binits{Y.~N.}}
(\byear{2016}).
\btitle{Performance bounds for parameter estimates of high-dimensional linear
  models with correlated errors}.
\bjournal{Electron. J. Stat.}
\bvolume{10}
\bpages{352--379}.
\bdoi{10.1214/16-EJS1108}
\bmrnumber{3466186}
\end{barticle}
\endbibitem

\bibitem[\protect\citeauthoryear{Wu and Zhao}{2007a}]{MR2323759}
\begin{barticle}[author]
\bauthor{\bsnm{Wu},~\bfnm{Wei~Biao}\binits{W.~B.}} \AND
  \bauthor{\bsnm{Zhao},~\bfnm{Zhibiao}\binits{Z.}}
(\byear{2007}a).
\btitle{Inference of trends in time series}.
\bjournal{J. R. Stat. Soc. Ser. B Stat. Methodol.}
\bvolume{69}
\bpages{391--410}.
\bdoi{10.1111/j.1467-9868.2007.00594.x}
\bmrnumber{2323759}
\end{barticle}
\endbibitem

\bibitem[\protect\citeauthoryear{Wu and Zhao}{2007b}]{zhaowu07}
\begin{barticle}[author]
\bauthor{\bsnm{Wu},~\bfnm{Wei~Biao}\binits{W.~B.}} \AND
  \bauthor{\bsnm{Zhao},~\bfnm{Zhibiao}\binits{Z.}}
(\byear{2007}b).
\btitle{Inference of trends in time series}.
\bjournal{J. R. Stat. Soc. Ser. B Stat. Methodol.}
\bvolume{69}
\bpages{391--410}.
\bdoi{10.1111/j.1467-9868.2007.00594.x}
\bmrnumber{2323759}
\end{barticle}
\endbibitem

\bibitem[\protect\citeauthoryear{Wu and Zhou}{2011}]{MR2827528}
\begin{barticle}[author]
\bauthor{\bsnm{Wu},~\bfnm{Wei~Biao}\binits{W.~B.}} \AND
  \bauthor{\bsnm{Zhou},~\bfnm{Zhou}\binits{Z.}}
(\byear{2011}).
\btitle{Gaussian approximations for non-stationary multiple time series}.
\bjournal{Statist. Sinica}
\bvolume{21}
\bpages{1397--1413}.
\bdoi{10.5705/ss.2008.223}
\bmrnumber{2827528}
\end{barticle}
\endbibitem

\bibitem[\protect\citeauthoryear{Zaitsev}{2000}]{MR1968723}
\begin{barticle}[author]
\bauthor{\bsnm{Zaitsev},~\bfnm{A.~Yu.}\binits{A.~Y.}}
(\byear{2000}).
\btitle{Multidimensional version of a result of {S}akhanenko in the invariance
  principle for vectors with finite exponential moments. {I}}.
\bjournal{Teor. Veroyatnost. i Primenen.}
\bvolume{45}
\bpages{718--738}.
\bdoi{10.1137/S0040585X97978555}
\bmrnumber{1968723}
\end{barticle}
\endbibitem

\bibitem[\protect\citeauthoryear{Zaitsev}{2001a}]{MR1971831}
\begin{barticle}[author]
\bauthor{\bsnm{Zaitsev},~\bfnm{A.~Yu.}\binits{A.~Y.}}
(\byear{2001}a).
\btitle{Multidimensional version of a result of {S}akhanenko in the invariance
  principle for vectors with finite exponential moments. {III}}.
\bjournal{Teor. Veroyatnost. i Primenen.}
\bvolume{46}
\bpages{744--769}.
\bdoi{10.1137/S0040585X97979305}
\bmrnumber{1971831}
\end{barticle}
\endbibitem

\bibitem[\protect\citeauthoryear{Zaitsev}{2001b}]{MR1978667}
\begin{barticle}[author]
\bauthor{\bsnm{Zaitsev},~\bfnm{A.~Yu.}\binits{A.~Y.}}
(\byear{2001}b).
\btitle{Multidimensional version of a result of {S}akhanenko in the invariance
  principle for vectors with finite exponential moments. {II}}.
\bjournal{Teor. Veroyatnost. i Primenen.}
\bvolume{46}
\bpages{535--561}.
\bdoi{10.1137/S0040585X97979123}
\bmrnumber{1978667}
\end{barticle}
\endbibitem

\bibitem[\protect\citeauthoryear{Zhou and Wu}{2010}]{MR2758526}
\begin{barticle}[author]
\bauthor{\bsnm{Zhou},~\bfnm{Zhou}\binits{Z.}} \AND
  \bauthor{\bsnm{Wu},~\bfnm{Wei~Biao}\binits{W.~B.}}
(\byear{2010}).
\btitle{Simultaneous inference of linear models with time varying
  coefficients}.
\bjournal{J. R. Stat. Soc. Ser. B Stat. Methodol.}
\bvolume{72}
\bpages{513--531}.
\bdoi{10.1111/j.1467-9868.2010.00743.x}
\bmrnumber{2758526}
\end{barticle}
\endbibitem

\end{thebibliography}


\newpage

\setcounter{page}{1}

\begin{supplement}[id=suppA]
  \sname{Online Supplementary}
  \slink[doi]{COMPLETED BY THE TYPESETTER}
  \sdatatype{.pdf}
  \sdescription{The online supplementary material contains the detailed proofs of Theorem \ref{th:main theorem} and some useful lemmas. The long detailed steps are in section \ref{sec:proofdetails} and the lemmas are postponed to section \ref{sec:sul}.
}
 \end{supplement}

\section{Detailed Steps of the Proof of Theorem 2.1}\label{sec:proofdetails}

\subsection{Preparation stage:}\label{ssc:prep stage}
The preparation stage consists of truncation approximation, $m$-dependence approximation and blocking approximation.

\subsubsection{Truncation approximation:}

\noindent Truncation approximation is necessary to allow higher moments manipulations. For $b>0$ and $v = (v_1, \ldots, v_d)\tran \in \mathbb{R}^d$, define
\begin{eqnarray}\label{eq:Truncation operator}
T_b(v)=(T_b(v_1),\ldots, T_b(v_d))\tran, \mbox{ where } T_b(w) = \min (\max(w,-b),b).
\end{eqnarray}

\begin{proposition}\label{prop:trunc}
Assume Condition (2.A). It is possible to choose a sequence $t_n \to 0$ slow enough such that we have
\begin{eqnarray}\label{eq:Approximation error 1}
\displaystyle\max_{1\leq i\leq n} |S_i- S_i^{\oplus} |= o_{P}(n^{1/p}), \mbox{ where } S_l^{\oplus} =\displaystyle\sum_{i=1}^{l} [T_{t_n n^{1/p}}(X_i)- E T_{t_n n^{1/p}}(X_i)].
\end{eqnarray}
\end{proposition}

\begin{proof} of Proposition \ref{prop:trunc}.
We introduce a very slowly converging sequence $t_n \to 0$ based on the uniform integrability condition (2.A). For every $t>0$, we have
\begin{equation}\label{eq:trunc}
\sup_{i} \frac{1}{t^p}E(|X_i|^p{\bf 1}_{|X_i|>tn^{1/p}})=0
 \mbox{ and } n \sup_{i} E \min ( \frac{|X_i|^\gamma}{t^{\gamma}n^{\gamma/p}}, 1 ) \rightarrow 0 \mbox{ as } n \to \infty,
\end{equation}
where $\gamma > p$. The second relation follows from Lemma \ref{lem:important inequality}. Clearly (\ref{eq:trunc}) implies that 
\begin{equation}\label{eq:tn part1}
\sup_{i} \frac{1}{t_n^p}E (|X_i|^p{\bf 1}_{|X_i|>t_n n^{1/p}}) 
 + n \sup_{i} E \min ( \frac{|X_i|^\gamma}{t_n^{\gamma}n^{\gamma/p}}, 1 ) \rightarrow 0 \mbox{ as } n \to \infty,
\end{equation}
holds for a sequence $t_n \to 0$ very slowly. Without loss of generality we can let \begin{eqnarray} \label{eq:tn part3 dup}
t_n \log \log n \to \infty
\end{eqnarray}
since otherwise we can replace $t_n$ by $\max(t_n, (\log\log n)^{-1/2} ) $ (say). The truncation operator $T_b$ in (\ref{eq:Truncation operator}) is Lipschitz continuous with Lipschitz constant 1. Let
\begin{eqnarray}\label{eq:partial sum of truncated}
R_{c,l}=\displaystyle\sum_{i=1+c}^{l+c}X_i^{\oplus}=\displaystyle\sum_{i=1+c}^{l+c} [T_{t_n n^{1/p}}(X_i)- E T_{t_n n^{1/p}}(X_i)].
\end{eqnarray}

By (\ref{eq:tn part1}), we have $P( \max_{i \leq n} |S_i- \sum_{j=1}^i T_{t_n n^{1/p}}(X_j) | = 0 ) \to 1$ in view of
\begin{eqnarray*}
\sup_j P\left(|X_j|>t_nn^{1/p}\right)\leq \sup_j \frac{1}{n t_n^p}E\left(|X_j|^p I\left(|X_j|>t_n n^{1/p}\right)\right)= o(1/n).
\end{eqnarray*}
Also by (\ref{eq:tn part1}), $\max_{j \le n} |E (X_j- T_{t_nn^{1/p}}(X_j))| = o(n^{1/p-1})$. Hence (\ref{eq:Approximation error 1}) follows.
\end{proof}

\subsubsection{$m$-dependence approximation:} The $m$-dependence approximation is a very important tool that is extensively used in literature; see for example the Gaussian approximation in Liu and Lin (2009, \cite{MR2485027}) and Berkes, Liu and Wu (2014, \cite{MR3178474}). For a suitably chosen sequence $m$, we look at the conditional mean $E(X_i|\epsilon_i, \ldots \epsilon_{i-m})$. This gives a very simple yet effective way to handle the original process in terms of a collection of $\epsilon_i$'s. Define the partial sum process
\begin{equation}\label{eq:partial sum}
\tilde{R}_{c,l}=\displaystyle\sum_{i=1+c}^{l+c} \tilde{X}_j, \mbox{ where }
\tilde{X}_{j}=E(T_{t_nn^{1/p}}(X_j)|\epsilon_{j}, \ldots, \epsilon_{j-m})- E(T_{t_nn^{1/p}}(X_j)).
\end{equation}
Write $\tilde{R}_{0,i} = \tilde{S}_i.$ From Lemma A1 in Liu and Lin (2009, \cite{MR2485027}), we have
\begin{eqnarray}\label{eq:Approximation for Ri}
\|\displaystyle\max_{1\leq l \leq n}|S^{\oplus}_l-\tilde{S}_{l}| \|_r \leq c_r n^{1/2} \Theta_{1+m,r}.
\end{eqnarray}
\noindent The proofs in \cite{MR2485027} are for stationary processes. Since our $\delta_{j, r}$ in (\ref{eq:fdm}) is defined in an uniform manner, the proof goes through for the non-stationary case as well. Assume
\begin{eqnarray}\label{eq:first condition}
n^{1/2-1/r} \Theta_{m,r} \to 0.
\end{eqnarray} 
By (\ref{eq:Approximation for Ri}) and (\ref{eq:first condition}), we have $n^{1/r}$ convergence in the $m$-dependence approximation step
\begin{eqnarray}\label{eq:Approximation error 2}
\displaystyle\max_{1\leq i\leq n} |S_i^{\oplus}-\tilde{S}_i|= o_P(n^{1/r}).
\end{eqnarray}


\ignore{
\noindent Define functional dependence measure for the truncated process $(T_{t_nn^{1/p}}(X_i))_{i \leq n}$ as 
\begin{eqnarray*}
\delta_{j,l}^{\oplus}=\sup_{i} \| T_{t_nn^{1/p}}(X_i)- T_{t_nn^{1/p}}(X_{i,(i-j)})\|_l, \mbox{ where } l \ge 2.
\end{eqnarray*}
Similarly, define the functional dependence measure for the $m$-dependent process $(\tilde{X}_{i})$ as \begin{eqnarray*}
\tilde{\delta}_{j,l}= \sup_{i} \| \tilde{X}_{i}-\tilde{X}_{i,(i-j)} \|_l.
\end{eqnarray*}

\noindent For these dependence measures, the following inequality holds for all $l \ge 2$:
\begin{eqnarray}\label{eq:fdm inequality}
\tilde{\delta}_{j,l} \leq \delta_{j,l}^{\oplus} \leq \delta_{j,l}.
\end{eqnarray}
}

\ignore{Proposition \ref{prop:block} gives the blocking approximation result.}

\subsubsection{Blocking approximation:} 
Towards the blocking approximation, we approximate the partial sum process $\tilde S_i$ by sums of $A_{j}$ where, for $j \ge 0$,
\begin{eqnarray}\label{eq:define big block}
A_{j+1}=\sum_{i=2jk_0m+1}^{(2k_0j+2k_0)m} \tilde{X}_i, \mbox{ where } k_0= \lfloor \Theta_{0,2}^2/\lambda_*\rfloor+2.
\end{eqnarray}
To this end, we will need the following two conditions, for some $\gamma>p$,
\begin{eqnarray}\label{eq:second condition}
n^{1-\gamma/r} m^{\gamma/2-1}\to 0,
\end{eqnarray}
\begin{eqnarray}\label{eq:third condition}
n^{1/p-1/\gamma} \sum_{j=m+1}^{\infty} \delta_{j,p}^{p/\gamma} \to 0.
\end{eqnarray}

We now define functional dependence measure for the truncated process $(T_{t_nn^{1/p}}(X_i))_{i \leq n}$ as 
\begin{eqnarray*}
\delta_{j,l}^{\oplus}=\sup_{i} \| T_{t_nn^{1/p}}(X_i)- T_{t_nn^{1/p}}(X_{i,(i-j)})\|_l, \mbox{ where } l \ge 2.
\end{eqnarray*}
Similarly, define the functional dependence measure for the $m$-dependent process $(\tilde{X}_{i})$ as \begin{eqnarray*}
\tilde{\delta}_{j,l}= \sup_{i} \| \tilde{X}_{i}-\tilde{X}_{i,(i-j)} \|_l.
\end{eqnarray*}

\noindent For these dependence measures, the following inequality holds for all $l \ge 2$:
\begin{eqnarray}\label{eq:fdm inequality}
\tilde{\delta}_{j,l} \leq \delta_{j,l}^{\oplus} \leq \delta_{j,l}.
\end{eqnarray}
We now proceed to proving Proposition \ref{prop:block}, the blocking approximation result. As mentioned in the main text, we need to assume conditions (\ref{eq:second condition}) and (\ref{eq:third condition}) for this step. The almost-polynomial rate of $m$ sequence as mentioned in 
(\ref{eq:define m dup}) is also assumed.

Remark: We need another condition for the blocking approximation (see (\ref{eq:third term simplify}) in the proof of Lemma \ref{lem:rosenthal lemma}). However, we skip it here and choose $m$ and $\gamma$ such that conditions (\ref{eq:first condition}), (\ref{eq:second condition}) and (\ref{eq:third condition}) are met. These will automatically imply this fourth one in view of (\ref{eq:form of thetaip}).

We assume an almost polynomial rate for $m$ sequence: for some $0 < L < 1$,
\begin{eqnarray}\label{eq:define m dup}
 m = \lfloor n^Lt_n^k \rfloor, \quad 0<k<(\gamma-p)/(\gamma/2-1).
\end{eqnarray}

\begin{proposition}\label{prop:block}
Assume (\ref{eq:second condition}) and (\ref{eq:third condition}) for some $\gamma>p$. Moreover, assume (\ref{eq:define m dup}) for the $m$ sequence and (\ref{eq:form of thetaip}) for the decay rate of $\Theta_{i,p}$ with some $A>\gamma/p$. Then
\begin{eqnarray}\label{eq:Approximation error 3}
\displaystyle\max_{1 \leq i \leq n} |\tilde{S}_i- S_i ^{\diamond}|= o_P(n^{1/r}), \mbox{ where } S_i^{\diamond}= \sum_{j=1}^{q_i} A_j, \,\, q_i=  \lfloor i/(2k_0m) \rfloor.
\end{eqnarray}
\end{proposition}

\begin{proof} of Proposition \ref{prop:block}: 
Let $\mathcal{S}= \{2ik_0m, 0 \le i \le q_n\}$, $\phi_n = (n^{1-\gamma/r} m^{\gamma/2-1})^{1/(2\gamma)}$. Then
\begin{eqnarray*}
P\left( \displaystyle\max_{1\leq l\leq n} |\tilde{R}_{0,l}-\displaystyle\sum_{j=1}^{\lfloor l/(2k_0m) \rfloor} A_{j} | \geq \phi_n n^{1/r}  \right) &\leq& \frac{n}{2k_0m}  \max_{c \in \mathcal{S}}P( \displaystyle \max_{1 \leq l\leq 2k_0m}|\tilde{R}_{c,l}| \geq \phi_n n^{1/r} )\\  \nonumber
& \leq& n\max_{c \in \mathcal{S}} \frac{E(\max_{1\leq l \leq 2k_0m}|\tilde{R}_{c,l}|^{\gamma})}{2k_0m\phi_n^{\gamma}n^{\gamma/r}} = O(\phi_n^\gamma),
\end{eqnarray*}
\noindent from the assumption (\ref{eq:second condition})  and Lemma \ref{lem:rosenthal lemma}. Since $\phi_n \to 0$, (\ref{eq:Approximation error 3}) follows.
\end{proof}

\ignore{
\begin{proof}
Let $\mathcal{S}= \{2ik_0m, 0 \le i \le q_n\}$ and $\phi_n = (n^{1-\gamma/r} m^{\gamma/2-1})^{1/(2\gamma)}$. Then
\begin{eqnarray*}
P\left( \displaystyle\max_{1\leq l\leq n} |\tilde{R}_{0,l}-\displaystyle\sum_{j=1}^{\lfloor l/(2k_0m) \rfloor} A_{j} | \geq \phi_n n^{1/r}  \right) &\leq& \frac{n}{2k_0m}  \max_{c \in \mathcal{S}}P( \displaystyle \max_{1 \leq l\leq 2k_0m}|\tilde{R}_{c,l}| \geq \phi_n n^{1/r} )\\  \nonumber
& \leq& n\max_{c \in \mathcal{S}} \frac{E(\max_{1\leq l \leq 2k_0m}|\tilde{R}_{c,l}|^{\gamma})}{2k_0m\phi_n^{\gamma}n^{\gamma/r}} = O(\phi_n^\gamma),
\end{eqnarray*}
\noindent from the assumption (\ref{eq:second condition})  and Lemma \ref{lem:rosenthal lemma}. Since $\phi_n \to 0$, (\ref{eq:Approximation error 3}) follows.
\end{proof}
}

\noindent Summarizing (\ref{eq:Approximation error 1}), (\ref{eq:Approximation error 2}) and (\ref{eq:Approximation error 3}), we can work on $S_i^{\diamond}$ in view of
\begin{eqnarray}\label{eq:Approximation error 4}
\displaystyle\max_{1 \leq i \leq n} |S_i- S_i^{\diamond}|= o_P(n^{1/r}).
\end{eqnarray}

\label{sec:tmb}

In the next two subsections we shall provide details of the arguments for steps mentioned in sections \ref{sec:caga} and \ref{sec:rcr}. section \ref{sec:cga} presents the conditional Gaussian approximation, where we shall apply Proposition \ref{re:zaitsev result} stated in section \ref{sec:sul}. section \ref{sec:uncga} deals with unconditional Gaussian approximation and regrouping.

\subsection{Conditional Gaussian approximation:}
\label{sec:cga}
The blocks $A_j$ created in (\ref{eq:define big block}) after the blocking approximation are weakly independent; except they share some dependence on the border. In this subsection, we look at the conditional process given the $\epsilon_i$ the blocks share in their borders. Demeaning the conditional process, we apply the Proposition \ref{re:zaitsev result} for the Gaussian approximation. For $1 \leq i \leq n$, let $\tilde{H}_i$ be a measurable function such that 
\begin{eqnarray}
\tilde{X}_{i}= \tilde{H}_{i} (\epsilon_{i},\ldots, \epsilon_{i-m}).
\end{eqnarray}

\noindent Recall Proposition \ref{prop:block} for the definition of $q_i$. Let $q = q_n$. For $j=1,\ldots, q$, define 
\begin{eqnarray*}
\bar{a}_{2k_0j}= \{a_{(2k_0j-1)m+1}, \ldots, a_{2k_0jm} \}
\mbox{ and }
a= \{\ldots, \bar{a}_0,\bar{a}_{2k_0},\bar{a}_{4k_0},\ldots  \}.
\end{eqnarray*}
Given $a$, define, for $2k_0jm+1 \leq i \leq (2k_0j+1)m$,
\begin{eqnarray*}
\tilde{X}_i(\bar{a}_{2k_0j})= \tilde{H}_i(\epsilon_i, \ldots, \epsilon_{2k_0jm+1},a_{2k_0jm}, \ldots ,a_{i-m})
\end{eqnarray*}
and for $(2k_0j+2k_0-1)m+1 \leq i \leq (2k_0j+2k_0)m$,
\begin{eqnarray*}
\tilde{X}_i(\bar{a}_{2k_0j+2k_0})= \tilde{H}_i(a_i, \ldots, a_{(2k_0j+2k_0-1)m+1},\epsilon_{(2k_0j+2k_0-1)m}, \ldots, \epsilon_{i-m}).
\end{eqnarray*}
Further, define the blocks as following, 
\begin{eqnarray}\label{eq:define blocks}
 F_{4j+1}(\bar{a}_{2k_0j})&=& \displaystyle\sum_{i=2k_0jm+1}^{(2k_0j+1)m} \tilde{X}_i(\bar{a}_{2k_0j}), \\ \nonumber
  F_{4j+2}&=& \displaystyle\sum_{i=(2k_0j+1)m+1}^{(2k_0j+k_0)m} \tilde{X}_i, \quad F_{4j+3} = \sum_{i=(2k_0j+k_0)m+1}^{(2k_0j+2k_0-1)m} \tilde{X}_i, \\ \nonumber
  F_{4j+4}(\bar{a}_{2k_0j+2k_0})&=& \displaystyle\sum_{i=(2k_0j+2k_0-1)m+1}^{(2k_0j+2k_0)m} \tilde{X}_i(\bar{a}_{2k_0j+2k_0}). \\ \nonumber
\end{eqnarray}
Similarly, for $j=1,\ldots, q$, define 
\begin{eqnarray*}
\bar{\vartheta}_{2k_0j}= \{\epsilon_{(2k_0j-1)m+1}, \ldots, \epsilon_{2k_0jm} \} \mbox{ and } 
\vartheta= \{\ldots, \bar{\vartheta}_0,\bar{\vartheta}_{2k_0},\bar{\vartheta}_{4k_0},\ldots  \}.
\end{eqnarray*}
Recall $A_j$ from (\ref{eq:define big block}). We have
\begin{eqnarray*}
A_{j+1}= F_{4j+1}(\bar{\vartheta}_{2k_0j})+F_{4j+2}+F_{4j+3}+F_{4j+4}(\bar{\vartheta}_{2k_0j+2k_0}).
\end{eqnarray*}
\noindent  Define the mean functions 
\begin{equation*} \Lambda_{4j+1}(\bar{a}_{2k_0j})=E^*(F_{4j+1}(\bar{a}_{2k_0j})) \text{ and } \Lambda_{4j+4}(\bar{a}_{2k_0j+2k_0})=E^*(F_{4j+4}(\bar{a}_{2k_0j+2k_0})),
\end{equation*}
where $E^*$ refers to the conditional moment given $a$. In the sequel, with slight abuse of notation, we will simply use the usual $E$ to denote moments of random variables conditioned on $a$. Introduce the centered process 
\begin{eqnarray}\label{eq:Demeaned process}
Y_j(\bar{a}_{2k_0j},\bar{a}_{2k_0j+2k_0})&=& F_{4j+1}(\bar{a}_{2k_0j})- \Lambda_{4j+1}(\bar{a}_{2k_0j})+F_{4j+2} \\ \nonumber
&&+F_{4j+3}+ F_{4j+4}(\bar{a}_{2k_0j+2k_0})- \Lambda_{4j+4}(\bar{a}_{2k_0j+2k_0}).
\end{eqnarray}
Following the definition of $S_n^{\diamond}$, we let 
\begin{equation*}
S_i(a)= \displaystyle\sum_{j=0}^{q_i-1} Y_j(\bar{a}_{2k_0j},\bar{a}_{2k_0j+2k_0}). 
\end{equation*}
The mean and variance function of $S_i(a)$ are respectively denoted by
\begin{eqnarray*}
M_i(a)&=& \displaystyle\sum_{j=0}^{q_i-1} [\Lambda_{4j+1}(\bar{a}_{2k_0j}) +\Lambda_{4j+4}(\bar{a}_{2k_0j+2k_0})], \\ \nonumber
Q_i(a) &=& \displaystyle\sum_{j=0}^{q_i-1} V_j(\bar{a}_{2k_0j},\bar{a}_{2k_0j+2k_0}), 
\end{eqnarray*}
where $V_j(\bar{a}_{2k_0j},\bar{a}_{2k_0j+2k_0})$ is the dispersion matrix of  $Y_j(\bar{a}_{2k_0j},\bar{a}_{2k_0j+2k_0})$. Define  
\begin{eqnarray}\label{eq:define Vj0}
V_{j0}(\bar{a}_{2 k_0j})&=& E(F_{4j-2}F_{4j-1}\tran + F_{4j-1}F_{4j-2}\tran)+Var(F_{4j-1}+ F_{4j}(\bar{a}_{2k_0j})- \Lambda_{4j}(\bar{a}_{2k_0j})) \cr
&& +Var(F_{4j+1}(\bar{a}_{2k_0j})- \Lambda_{4j+1}(\bar{a}_{2k_0j})+F_{4j+2}).
\end{eqnarray}
Note that, the following identity holds for all $t$:
\begin{eqnarray}\label{eq:identity of variance}
\displaystyle\sum_{j=0}^{t} V_j(\bar{a}_{2k_0j},\bar{a}_{2k_0j+2k_0}) = L(\bar{a}_{0})+\displaystyle\sum  _{j=1}^{t-1} V_{j0}(\bar{a}_{2k_0j})+U_t(\bar{a}_{2k_0t+2k_0}),
\end{eqnarray}
where $L(\bar{a}_{0})=Var(F_1(\bar{a}_{0})+F_2)$ and 
\begin{equation}\label{eq:Uj define}
U_{t-1}(\bar{a}_{2k_0t})= E(F_{4t-2}F_{4t-1}\tran + F_{4t-1}F_{4t-2}\tran)+Var(F_{4t-1}+ F_{4t}(\bar{a}_{2k_0t})- \Lambda_{4t}(\bar{a}_{2k_0t})).
\end{equation}
Define $$L_{\gamma}^{a}= \displaystyle\sum_{j=0}^{q-1} E(| Y_j(\bar{a}_{2k_0j},\bar{a}_{2k_0j+2k_0})|^{\gamma}).$$ 

\noindent In the sequel, we suppress $Y_j(\bar{a}_{2k_0j},\bar{a}_{2k_0j+2k_0})$, $Y_j(\bar{\vartheta}_{2k_0j},\bar{\vartheta}_{2k_0j+2k_0})$, $V_j(\bar{a}_{2k_0j}, \bar{a}_{2k_0j+2k_0})$, $V_{j0}(\bar{a}_{2k_0j})$, $V_j(\bar{\vartheta}_{2k_0j}, \bar{\vartheta}_{2k_0j+2k_0})$ and $V_{j0}(\bar{\vartheta}_{2k_0j})$ as just $Y_j^{a}$,$Y_j^{\vartheta}$, $V_j^{a}$, $V_{j0}^{a}$, $V_j^{\vartheta}$ and $V_{j0}^{\vartheta}$ respectively. We apply Proposition \ref{re:zaitsev result} to the independent mean zero random vectors $Y_j^a$. 

Proposition \ref{re:zaitsev result} concerns Gaussian approximation for independent vectors. There are several types of Gaussian approximations in literature for independent vectors. We find the following result by G{\"o}tze and Zaitsev (2008, \cite{MR2760567}) particularly useful since it provides an explicit and good approximation  bound  for  the  partial  sums. This has been used several times in our proof.
\begin{proposition}\label{re:zaitsev result}
Let $\xi_1,\ldots, \xi_n$ be independent $\mathbb{R}^d$-valued mean zero random vectors. Assume that there exist $s \in {\mathbb N}$ and a strictly increasing sequence of non-negative integers $\eta_0=0<\eta_1< \ldots < \eta_s=n$ satisfying the following conditions. Let
$$\zeta_k=\xi_{\eta_{k-1}+1}+\ldots +\xi_{\eta_k}, \quad Var (\zeta_k)= B_k, \quad k= 1,\ldots,s$$
and $L_{\gamma}= \sum_{j=1}^n E(| \xi_j |^{\gamma})$, $\gamma \ge 2$, and assume that, for all $k= 1,\ldots,s$, 
\begin{eqnarray}\label{eq:zaitsev condition 1}
C_1 w^2  \leq \rho_*(B_k) \leq \rho^*(B_k) \leq  C_2 w^2,
\end{eqnarray}
\noindent where $w=(L_{\gamma})^{1/\gamma}/\log^* s$, with some positive constants $C_1$ and $C_2$. Suppose the quantities $$\lambda_{k,\gamma}= \displaystyle\sum_{j=\eta_{k-1}+1}^{\eta_k} E \|\xi_j \|^{\gamma}, \hspace{0.1 in} k= 1, \ldots s, $$
\noindent satisfy, for some $0  <\epsilon <1$ and constant $C_3$,
\begin{eqnarray}\label{eq:zaitsev condition 2}
C_3 d^{\gamma/2}s^{\epsilon} (\log^*s)^{\gamma+3} \max_{1 \leq k \leq s} \lambda_{k,\gamma} \leq L_{\gamma}.
\end{eqnarray}
Then one can construct on a probability space independent random vectors $X_1,\ldots ,X_n$ and a corresponding set of independent Gaussian vectors $Y_1, \ldots ,Y_n$ so that $(X_j)_{j=1}^n \stackrel{D}{=}(\xi_j)_{j=1}^n$, $E(Y_j)=0$, $Var(Y_j)= Var(X_j), 1\le j \le n$, and for any $z>0$,
\begin{eqnarray*}\label{eq:zaitsev approximation}
P\left(\displaystyle\max_{t \leq n} |\displaystyle\sum_{i=1}^t X_i - \displaystyle\sum_{i=1}^t Y_i | \geq z\right) \leq C_*L_{\gamma}z^{-\gamma}.
\end{eqnarray*}
\noindent where $C_*$ is a constant that depends on $d,\gamma,C_1,C_2$ and $C_3$. 
\end{proposition}

We need to find a suitable sequence $\eta_k$ that allows us to get constants $C_1,C_2$ in (\ref{eq:zaitsev condition 1}) and $C_3$ in (\ref{eq:zaitsev condition 2}). There are roughly $q=n/(2k_0m)$ many $Y_j^{a}$ random variables. Define
\begin{eqnarray}\label{eq:choice of l}
l = \lfloor q^{2/\gamma}/\log ^2 q \rfloor.
\end{eqnarray}
To apply Proposition \ref{re:zaitsev result}, we choose the sequence $\eta_k= kl$ and $s \asymp q/l$. This choice is justified by proving the following series of propositions.


\begin{proposition}\label{prop:Aj bounded}
Recall $\lambda_*$ and $A_j$ from (2.B) and (\ref{eq:define big block}) respectively. There exists a constant $\delta>0$ such that 
\begin{eqnarray*}
2(\lambda_*+\delta) k_0m \leq \rho_*(Var (A_j)) \leq \rho^*(Var (A_j)) \leq  \|A_j\|^2 \leq 2k_0m\Theta_{0,2}^2 .
\end{eqnarray*} 
\end{proposition}

\begin{proposition}
\label{prop:lower bound on Var Yj'}
We can get positive constants $c_1$ and $c_2$ such that for all $j$,
\begin{eqnarray}\label{eq:prop 3.4 eq2}
c_1m \leq \rho_*(Var(Y_j^{\vartheta})) \leq \rho^*(Var(Y_j^{\vartheta})) \leq E(|Y_j^{\vartheta}|^2)\leq c_{2}m.
\end{eqnarray}

\end{proposition}

\begin{proposition}\label{prop:vja}
For $l$ in (\ref{eq:choice of l}), there exists constant $c_3$ such that,
$$P\left(\max_{1 \leq t \leq q/l}| Var\left(\sum_{j=(t-1)l}^{tl-1} Y_j^{a}\right)- E\left(Var\left(\sum_{j=(t-1)l}^{tl-1} Y_j^{a}\right)\right)| \ge c_3lm\right) \to 0.$$ 
\end{proposition}

\begin{proposition}\label{prop:Lgamma}
We can get constants $c_4$ and $c_5$ such that  

$$P(c_4q^{2/\gamma}m \leq (L_{\gamma}^{a} )^{2/\gamma} \leq c_5q^{2/\gamma}m) \to 1.$$ 
\end{proposition}

\begin{proposition}\label{prop:choice of l}
Choose $\eta_k=kl$ with $l$ being defined in (\ref{eq:choice of l}). Then we can get $C_1$ and $C_2$ such that (\ref{eq:zaitsev condition 1}) is satisfied. Moreover, with $l$ in (\ref{eq:choice of l}), we can get $C_3$ such that (\ref{eq:zaitsev condition 2}) holds.
\end{proposition}

\noindent Thus, we use Proposition \ref{re:zaitsev result} to construct $d$-variate mean zero normal random vectors $N^{a}_{j}$ and random vectors $E_{j}^{a}$ such that
\begin{equation*}
E_j^{a} \stackrel{D}{=} Y_j^a \mbox{ and } Var(N_{j}^{a})= Var(Y_{j}^a), \quad 0 \leq j \leq q-1, 
\end{equation*}
\begin{equation}\label{eq:finalapprox1}
 P_{a}\left( \displaystyle\max_{1 \leq i \leq n} | \Pi_i^{a}- D_i^{a} | \geq c_0z \right) \leq C { L_{\gamma}^{a} \over z^\gamma}, \mbox{ where }
  \Pi_i^{a}=\displaystyle\sum_{j=0}^{q_i-1} E_{j}^{a},\,\, 
D_i^{a}=\displaystyle\sum_{j=0}^{q_i-1} N_{j}^{a}
\end{equation}
and $C$ is a constant depending on $\gamma,c_1,\ldots,c_5$ and $C_3$. These constants are free of $a$. We can create a set ${\cal A}$ with $P({\cal A}) \to 1$ so that $a \in {\cal A}$ implies the statements in Proposition \ref{prop:Lgamma} and Proposition \ref{prop:vja} hold. Putting $z=n^{1 /r}$ above in (\ref{eq:finalapprox1}), by Lemma \ref{lem:rosenthal lemma} and the restriction (\ref{eq:equation 2}), we have, as $n \to \infty$,
\begin{eqnarray}\label{eq:proof of final approximation 1}
E(L_{\gamma}^{a}n^{-\gamma/r}) \leq  \frac{q}{n^{\gamma /r}}c_{\gamma} \max_cE( |\tilde{R}_{c,2k_0m}|^{\gamma})  
= O(n^{1-\gamma/r}m^{\gamma/2-1}) \rightarrow 0,
\end{eqnarray}
using
\begin{eqnarray*}
E(|Y_{j}(\bar{\vartheta}_{2k_0j},\bar{\vartheta}_{2k_0j+2k_0}) |^{\gamma}) \leq c_{\gamma} \max_c E(|\tilde{R}_{c,2k_0m}|^{\gamma}) = O(m^{\gamma/2}).
\end{eqnarray*}
Hence, conditioning on whether $a$ lies in ${\cal A}$ or not, from (\ref{eq:proof of final approximation 1}) we obtain,
\begin{eqnarray}\label{eq:proof of final approximation 3}
\displaystyle\max_{ i \leq n} | \Pi_i^{\vartheta}- D_i^{\vartheta} | = o_P(n^{1/r}).
\end{eqnarray}

\subsection{Unconditional Gaussian approximation and Regrouping:}
\label{sec:uncga}
Here we shall work with the processes $\Pi_i^{\vartheta}, \mu_i^{\vartheta}$ and $D_i^{\vartheta}$. Note that, $V_{j0}(\bar{a}_{2k_0j})$ defined in (\ref{eq:define Vj0}) is a function of $\vartheta$ and might not be positive definite in an uniform fashion. For a constant $0 < \delta_* < \lambda_*$, let
\begin{eqnarray}\label{eq:define Vj1}
V_{j1}(\bar{a}_{2k_0j})=
\begin{cases} 
V_{j0}(\bar{a}_{2k_0j}) &\mbox{if } \rho_*(V_{j0}^a) \ge \delta_*m,\\
(\delta_*m) I_d & \mbox{otherwise,}\end{cases}
\end{eqnarray}
\noindent which is a positive-definitized version of $V_{j0}(\bar{a}_{2k_0j})$. The following proposition shows that partial sums of $V_{j0}(\bar{a}_{2k_0j})$ and $V_{j1}(\bar{a}_{2k_0j})$ are close to each other. 

\begin{proposition}\label{prop:Vj final}
For some $\iota > 0$, we have
\begin{eqnarray*}
\max_{i \leq n}E \left(\left|\sum_{j=1}^{max(1,q_i-1) }(V_{j0}(\bar{a}_{2k_0j})-V_{j1}(\bar{a}_{2k_0j}))\right| \right)=o_P(n^{2/r-\iota}).
\end{eqnarray*}
\end{proposition}
\noindent Henceforth in the sequel we will slightly abuse $max(1,q_i-1)=max(1,\lfloor i/(2k_0m) \rfloor-1)$ and simply use $q_i-1=\lfloor i/(2k_0m) \rfloor-1$ for presentational clarity.  
\begin{proof} of Proposition \ref{prop:Vj final}.
Recall (\ref{eq:define blocks}) for the definition of $F_{4j+1}(.), F_{4j+2}$ etc. Define 
\begin{equation*}
F_{21}= \displaystyle\sum_{i=m+1}^{2m} \tilde{X}_i.
\end{equation*}
Define the projection operator $P_i$ by
\begin{eqnarray*}
P_i Y = E(Y |\mathcal{F}_i ) - E(Y |\mathcal{F}_{ i-1} ), \quad Y \in \mathcal{L}_1 .
\end{eqnarray*}
\noindent For $1 \le j \le m$, $\|P_jF_{21}\| \le \sum_{i=m+1-j}^{m}\delta_{i,2}$. Since $\| E(F_{21}\tran | {\cal F}_m) \|^2 = \sum_{j=1}^m \|P_j F_{21}\|^2$, we have
\begin{eqnarray}\label{eq:Cov random small}
|E(F_1(\bar{a}_0)F_2\tran)| &=& |E(F_1(\bar{a}_0)F_{21}\tran)|=
 |E(F_1(\bar{a}_0) E(F_{21}\tran | {\cal F}_m) )| \cr
& \le & \|F_1(\bar{a}_0)\| (\sum_{j=1}^m (\sum_{i=m+1-j}^{m}\delta_{i,2})^2)^{1/2}.
\end{eqnarray}
Under the decay condition on $\Theta_{i,p}$ in (\ref{eq:form of thetaip}), we have 
$$E(|E(F_{1}(\bar{a}_0)F_{21}\tran)|^{\gamma}) =O(m^{\max(\gamma/2,\gamma-\chi\gamma)}). $$
We expand the last term of $V_{j0}(\bar{a}_{2k_0j})$ (see (\ref{eq:define Vj0})). Also note that, 
$$|E(F_{4j-2}F_{4j-1}\tran)+E(F_{4j-1}F_{4j-2}\tran)| \ll m \text{ and } \rho_*(Var(F_{4j+2})) \ge (k_0-1)\lambda_*m.$$ \noindent Then Proposition \ref{prop:Vj final} follows from the fact that our solution of $\gamma$ from (\ref{eq:equation 1}), (\ref{eq:equation 2}), and (\ref{eq:equation 3}) satisfy $\gamma>\max(2,4\chi)$ for $\chi \leq \chi_0$ and
\begin{eqnarray*} 
n \max_j P \left(\rho_*(V_{j0}^a)< \delta_*m\right) &\le& 2n \max_j P(|E(F_{4j+1}(\bar{a}_{2k_0j})F_{4j+2}\tran)|\geq -\theta m/2) \\ 
&= & O(n) \frac{m^{\max(\gamma/2,\gamma-\chi \gamma)}}{m^{\gamma}} = o(n^{2/r-\iota}),
\end{eqnarray*}
for some $\iota>0$ since we can choose $\delta_*$ such that $\theta = (k_0-1)\lambda_* - \delta_* > 0$.
\end{proof}
\noindent Recall (\ref{eq:Uj define}) for the definition of $U_j$. By Lemma \ref{lem:rosenthal lemma} and  Jensen's inequality, we obtain $\max_j \| U_{j}(\bar{\vartheta}_{2k_0j+2k_0}) \|_{\gamma / 2} = O(m^{1/2})$. By (\ref{eq:equation 2}), $\phi_n := q^{1/\gamma} m^{1/2} n^{-1/r} \to 0$. Then
\begin{eqnarray*}
P \left( \displaystyle\max_{0 \leq j \leq q-1}|U_{j}(\bar{\vartheta}_{2k_0j+2k_0})| \geq \phi_n n^{2/r} \right) 
&\leq& \sum_{j=0}^{q-1} P \left( |U_{j}(\bar{\vartheta}_{2k_0j+2k_0})| \geq \phi_n n^{2/r} \right) \\ 
&=& O(\phi_n^{-\gamma/2} n^{1-\gamma/r} m^{\gamma/2-1}) = O(\phi_n^{\gamma/2}) \to 0. 
\end{eqnarray*}
Similarly, $|L(\bar{\vartheta}_0)|=o_P(n^{2/r})$. Thus, by (\ref{eq:identity of variance}) and Proposition \ref{prop:Vj final}, since $Var(Y_j^a) = Var(N_j^a)$, one can construct i.i.d. $N(0, I_d)$ normal vectors $Z_{l}^a, l \in \mathbb{Z}$, such that 
\begin{equation*}
\max_{i \leq n} |D_i^{\vartheta}-\varsigma_i(\vartheta)|=o_P(n^{1/r}),
 \mbox{ where } \varsigma_i(a)= \displaystyle\sum_{j=1}^{q_i-1}V_{j1}^{0}(\bar{a}_{2k_0j})^{1/2}Z_{j}^{a}.  
\end{equation*}
\noindent By (\ref{eq:proof of final approximation 3}), we have $$ \displaystyle\max_{i \leq n}|\Pi_i^{\vartheta}-\varsigma_i(\vartheta)|= o_P(n^{1/r}).$$
\noindent Let $Z_{l}^{*}, l \in \mathbb{Z}$, independent of $(\epsilon_j)_{j \in \mathbb{Z}}$, be i.i.d.  $N(0, I_d)$ and define 
$$\Psi_i= \displaystyle\sum_{j=1}^{q_i-1} V_{j1}(\bar{\vartheta}_{2k_0j})^{1/2} Z_{j}\noindent ^*.$$
\noindent From the distributional equality, 
\begin{eqnarray}\label{eq:distribution equality}
(\Pi_i^{\vartheta}+M_i(\vartheta))_{1\le i \le n} \stackrel{D}{=} (S_i^{\diamond})_{1 \le i \le n},
\end{eqnarray}
\noindent we need to prove Gaussian approximation for the process $\Psi_i + M_i(\vartheta).$ Define
\begin{eqnarray*}
B_j= V_{j1}(\bar{\vartheta}_{2k_0j})^{1/2} Z_{j}^*+ \Lambda_{4j}(\bar{\vartheta}_{2k_0j})+\Lambda_{4j+1}(\bar{\vartheta}_{2k_0j}),
\end{eqnarray*}
\noindent which are independent random vectors for $j=1,\ldots ,q$ and let 
\begin{eqnarray*}
S_i^{\sharp}= \displaystyle\sum_{j=1}^{q_i-1}B_j \mbox{ and } W_i^{\sharp}= \Psi_i+M_i(\vartheta)- S_i^{\sharp}.
\end{eqnarray*}
Note that,
\begin{eqnarray}\label{eq:last term}
\max_{i \leq n} |W_i^{\sharp}|=\max_{i \leq n}|\Lambda_{4q_i}(\vartheta_{2k_0q_i})+\Lambda_{1}(\vartheta_{0})|=o_P(n^{1/r}).
\end{eqnarray}
Conditions (\ref{eq:zaitsev condition 1}) and (\ref{eq:zaitsev condition 2}) can be verified easily with this unconditional process $(S)_i^{\sharp}$ to use the Proposition \ref{re:zaitsev result}. Thus, there exists $B_j^{new}$ and Gaussian random variable $B_j^{gau}$, such that $(B_j^{new})_{j \leq q-1} \stackrel{D}{=}(B_j)_{j \leq q-1} $ and corresponding $B_j^{gau} \sim N  ( 0, Var ( B_j))$, such that
\begin{eqnarray}\label{ eq:main theorem final 1}
\max_{i \leq n} |\sum_{j=1}^{\lfloor i/2k_0m \rfloor-1} B_j^{new}- \sum_{j=1}^{\lfloor i/2k_0m \rfloor-1} B_j ^ {gau} | &=& o_P(n^{1/r}). 
\end{eqnarray}
By (\ref{eq:Approximation error 3}), (\ref{eq:distribution equality}), (\ref{eq:last term}) and (\ref{ eq:main theorem final 1}), we can construct a process $S_i^c$ and $B_j^c$ such that $(S_i^c)_{i \leq n} \stackrel{D}{=} (S_i)_{i \leq n}$ and $(B_j^c)_{j \leq q-1} \stackrel{D}{=} (B_j^{gau})_{j \leq q-1}$ and 
\begin{eqnarray}\label{ eq:main theorem final 3}
\max_{i \leq n} |S_i^c - \sum_{j=1}^{\lfloor i/(2k_0m) \rfloor-1} B_j^c | &=& o_P(n^{1/r}). 
\end{eqnarray}
Relabel this final Gaussian process as 
\begin{eqnarray*}\label{eq:final approximation}
G_i^c= \sum_{j=1}^{ \lfloor i/2k_0m \rfloor-1} (Var (B_j))^{1/2}Y_j^c,
\end{eqnarray*}
where $Y_j^c$ are i.i.d. $N(0, I_d)$. This concludes the proof of Theorem \ref{th:main theorem}. \qed

\begin{proof} of Proposition \ref{prop:Aj bounded}. Without loss of generality, we prove it for $j=1$. Note that
\begin{eqnarray}\label{eq:mdep}
2k_0m \lambda_* \leq \rho_*(Var (S_{2k_0m})) \leq \rho^*(Var (S_{2k_0m}))  \leq  \|\sum_{i=1}^{2k_0m}X_i \|^2  \leq 2k_0m\Theta_{0,2}^2.
\end{eqnarray}

\noindent Recall $X_i^{\oplus}$ and $\tilde{X}_i$ from (\ref{eq:partial sum of truncated}) and (\ref{eq:partial sum}). The same upper bound works for $S_i^{\oplus}$ and $\tilde{S}_i$. Note that, $\|S_{2k_0m}^{\oplus}-S_{2k_0m}\|=o(m)$ and from \cite{MR2660298}, we have  $$\| A_{1}  - S_{2k_0m}^{\oplus}  \|= O(\sqrt{2k_0m} \Theta_{m,2} )=o(\sqrt{2k_0m}).$$ 
This concludes the proof using the Cauchy-Schwartz inequality.
\end{proof}

\begin{proof} of Proposition \ref{prop:lower bound on Var Yj'}.
As $A_j$ is the block sum of the $m$-dependent processes with length $2k_0m$, we have, using (\ref{eq:mdep}), for all $j$,  $$2k_0m (\lambda_*+\delta) \leq E(|A_j|^2) \leq 2k_0m \Theta_{0,2}^2,$$

\noindent for some small $\delta>0$. We conclude the proof by using 
\begin{equation*}
|E(|Y_j^{\vartheta}|^2)-E(|A_{j+1}|^2)|=|\Lambda_{4j+1}(\bar{\vartheta}_{2k_0j})|^2+|\Lambda_{4j+4}(\bar{\vartheta}_{2k_0j+2k_0})|^2  \leq 2m \Theta_{0,2}^2 
\end{equation*}
and $ k_0> \Theta_{0,2}^2/\lambda_* + 1$. Using similar arguments, (\ref{eq:prop 3.4 eq2}) follows. 
\end{proof}

\begin{proof} of Proposition \ref{prop:vja}.
Note that, without loss of generality, we can assume $V_j^a$ to be independent for different $j$ since otherwise we can always break the probability statement in even and odd blocks and prove the statement separately. We use Corollary 1.6 and Corollary 1.7 from Nagaev (1979, \cite{MR542129}) respectively for the case $\gamma<4$ and $\gamma \geq 4$ on $|V_j^a-E(V_j^a)|$ to deduce that it suffices to show the following
\begin{eqnarray}\label{eq:vja-evja small}
q\max_{1\leq t \leq q/l} \max_{t(l-1)+1\leq j\leq tl}P(|V_j^a-E(V_j^a)|\geq lm) \to 0.
\end{eqnarray}
We expand and write $V_j^a$ as follows:
\begin{eqnarray}\label{eq:Vja}
V_j^a&=& Var(F_{4j+1}(\bar{a}_{2k_0j})- \Lambda_{4j+1}(\bar{a}_{2k_0j})) +Var(F_{4j+2}+F_{4j+3}) \\
\nonumber &+&E((F_{4j+1}(\bar{a}_{2k_0j})- \Lambda_{4j+1}(\bar{a}_{2k_0j}))F_{4j+2}\tran)+ E(F_{4j+2}(F_{4j+1}(\bar{a}_{2k_0j})- \Lambda_{4j+1}(\bar{a}_{2k_0j}))\tran) \\  \nonumber &+& E(F_{4j+3}(F_{4j+4}(\bar{a}_{2k_0j+2k_0})- \Lambda_{4j+4}(\bar{a}_{2k_0j+2k_0}))\tran) \\\nonumber
&+&E((F_{4j+4}(\bar{a}_{2k_0j+2k_0})- \Lambda_{4j+4}(\bar{a}_{2k_0j+2k_0}))F_{4j+3}\tran) \\ \nonumber &+&Var(F_{4j+4}(\bar{a}_{2k_0j+2k_0})- \Lambda_{4j+4}(\bar{a}_{2k_0j+2k_0})).
\end{eqnarray}
Using derivation similar to (\ref{eq:Cov random small}), it suffices to show (\ref{eq:vja-evja small}) for only the first and last term in (\ref{eq:Vja}). Moreover, we assume $d=1$ and $j=1$ to simplify notations. The proofs and the theorems used can be easily extended to vector-valued processes. Denote by $\tilde{S}_{m,\{j\}}$ for the sum $\tilde{S}_m$ with $\epsilon_j$ replaced by an i.i.d. copy $\epsilon_j'$. For the first term, by Burkholder's inequality, 
\begin{eqnarray*}
E(|Var(F_1(\bar{a}_0))-E(Var(F_1(\bar{a}_0)))|^{\gamma/2} )
&=& E(|E(\tilde{S}_m^2|a_{1-m},\ldots,a_0)-E(\tilde{S}_m^2)|^{\gamma/2}) \\
&=& \|\sum_{j=-m}^{0} P_j \tilde{S}_m^2\|_{\gamma/2}^{\gamma/2}
\leq c_{\gamma} (\sum_{j=-m}^{0} \|P_j \tilde{S}_m^2\|_{\gamma/2}^2)^{\gamma/4}
\end{eqnarray*}
For $-m \le j \le 0$, $\|P_j \tilde{S}_m^2\|_{\gamma/2} \le \|\tilde{S}_m^2 - \tilde{S}_{m,\{j\}}^2 \|_{\gamma/2} \le \|\tilde{S}_m - \tilde{S}_{m,\{j\}}\|_\gamma \|\tilde{S}_m + \tilde{S}_{m,\{j\}}\|_\gamma$. Note that $\|\tilde{S}_m \|_\gamma = O(m^{1/2})$ and $\|\tilde{S}_m - \tilde{S}_{m,\{j\}}\|_\gamma \le \sum_{r=1}^{m} \tilde{\delta}_{r-j,\gamma}$. By Lemma \ref{lem:truncated deltajy}, $\tilde{\delta}_{k,\gamma} \leq 2 n^{1/p-1/\gamma}t_n^{1-p/\gamma} \delta_{k,p}^{p/\gamma}.$ Then since $3 > 2(\chi+1)p/\gamma$ for $\chi \leq \chi_0$, we have
\begin{eqnarray}\label{eq:first term part 2}
\sum_{j=-m}^{0} \|P_j \tilde{S}_m^2\|_{\gamma/2}^2 
&=& O(m) \sum_{j=-m}^{0} \sum_{r=1}^{m} (\tilde{\delta}_{r-j,\gamma})^2 \\  \nonumber 
&=& O(m) n^{2/p-2/\gamma}t_n^{2-2p/\gamma} \sum_{j=0}^{m} (\sum_{r=1}^{m} \delta_{r+j,p}^{p/\gamma})^2 \\ \nonumber
&=&O(m)n^{2/p-2/\gamma}t_n^{2-2p/\gamma}m^{3-2(\chi+1)p/\gamma}(\log m)^{-2Ap/\gamma},
\end{eqnarray}
by (\ref{eq:form of thetaip}) and the H\"older inequality. Then, since $A>2\gamma/p$ and $\log m \asymp \log q \asymp \log n$, 
\begin{eqnarray}\label{eq:1st term small}
&&qE(|Var(F_1(\bar{a}_0))-E(Var(F_1(\bar{a}_0)))|^{\gamma/2} ) \\ \nonumber &\lesssim& qm^{\gamma-(\chi+1)p/2}n^{\gamma/2p-1/2}t_n^{\gamma/2-p/2} (\log n)^{-Ap/2}= o((lm)^{\gamma/2}),
\end{eqnarray}
using (\ref{eq:tn part3 dup}), (\ref{eq:equation 3}) and the choice of $l$ in (\ref{eq:choice of l}). For the last term in (\ref{eq:Vja}), we view $E(F_4(\bar{a}_{2k_0})^2)$ as 
$$E(F_4(\bar{a}_{2k_0})^2)=E((\tilde{S}_{2k_0m}-\tilde{S}_{(2k_0-1)m})^2|a_{(2k_0-1)m+1},\ldots a_{2k_0m})$$ \noindent and show that it is close to $(\tilde{S}_{2k_0m}-\tilde{S}_{(2k_0-1)m})^2$. Let $\mathcal{F}_j^m = (\epsilon_j,\ldots,\epsilon_m)$. Note that,
\begin{eqnarray}\label{eq:vj 4th term diff small}
\hspace{-0.5 in}\|\tilde{S}_m^2-E(\tilde{S}_m^2|a_m,\ldots,a_1)\|_{\gamma/2}^{\gamma/2}& \lesssim & (\sum_{j=-m-1}^{0} \|E(\tilde{S}_m^2|\mathcal{F}_j^m)-E(\tilde{S}_m^2|\mathcal{F}_{j+1}^m) \|_{\gamma/2}^2 )^{\gamma/4} \\ \nonumber 
&\leq& c m^{\gamma-(\chi+1)p/2}n^{\gamma/2p-1/2}t_n^{\gamma/2-p/2}(\log m)^{-Ap/2} \\ \nonumber &=& o(q^{-1}(lm)^{\gamma/2}),
\end{eqnarray}
similar to the derivation in (\ref{eq:first term part 2}). By (\ref{eq:1st term small}) and (\ref{eq:vj 4th term diff small}), it suffices to show that 
\begin{eqnarray}\label{eq:vja-evja small 2}
\frac{n}{m} P(|\tilde{S}_m|\ge \sqrt{lm}) \to 0.
\end{eqnarray}
Using the Nagaev-type inequality from Wu and Wu  (2016, \cite{MR3466186}) we obtain
\begin{eqnarray}\label{eq:nagaev 22}
\hspace{-0.3 in}P(|\tilde{S}_m | \geq \sqrt{lm}) &\leq& C_1 \frac{m^{\max\{1,p(1/2-\chi)\} }}{(lm)^{p/2}}+C_2\exp(-C_3l),
\end{eqnarray}

\noindent where $C_1,C_2$ and $C_3$ depend on $\chi$ and $p$. The second term in (\ref{eq:nagaev 22}) is $o(m/n)$ since $e^{-l} \to 0$ very fast.  For the first term in (\ref{eq:nagaev 22}), if $\chi < 1/2- 1/p$, then
\begin{eqnarray*}
\frac{n}{m}\frac{m^{p(1/2-\chi)}}{(lm)^{p/2}}= (\log n)^{p}n^{1-p/\gamma+L(p/\gamma-p\chi-1)} t_n^{k(p/\gamma-p\chi-1)}=o(1) , 
\end{eqnarray*}
as from (\ref{eq:equation 3}) we have $1-p/\gamma+L(p/\gamma-p\chi-1)  =L(p/\gamma-1)(\chi p+p+1)<0.$ If $1/2-1/p \leq \chi < \chi_0$ and consequently $r<p$, then we have, for the first term in (\ref{eq:nagaev 22}),
\begin{eqnarray}\label{eq:nagaev difficult case}
\frac{n}{m}\frac{m}{(lm)^{p/2}}= (\log n)^{p} n^{p(1/p-1/\gamma+L(1/\gamma- 1/2))}t_n^{k(p/\gamma-p/2)}=o(1),
\end{eqnarray}
using (\ref{eq:tn part3 dup}), $r<p$ and the fact that $r$ satisfy $1/r-1/\gamma+L(1/\gamma-1/2)=0.$
\end{proof}

\begin{proof} of Proposition \ref{prop:Lgamma}. 
By Lemma \ref{lem:rosenthal lemma}, $E(L_{\gamma}^a) \asymp qm^{\gamma/2}$. Then it suffices to prove 
\begin{eqnarray}
 \label{eq:Lgamma1}
P(|L_{\gamma}^a-E(L_{\gamma}^a)|\geq c q m^{\gamma/2}/\log q) \to 0,
\end{eqnarray}
holds for some constant $c > 0$. Note that $E(|Y_j^a|^{\gamma})$ are even indices $j$ (also for odd indices $j$). Thus we can prove the statement separately by breaking $L_{\gamma}^a$ in sum of even and odd $E(|Y_j^{a}|^{\gamma})$. Without loss of generality, we assume all $E(|Y_{j}^a|^{\gamma})$ are independent and proceed. Define $J_j=(2k_0m)^{-\gamma/2}E(|\tilde{S}_{2k_0mj}-\tilde{S}_{2k_0m(j-1)}|^{\gamma}|\bar{a}_{2k_0(j-1)}, \bar{a}_{2k_0j})$ and $\theta=l^{\gamma/2}=q/(\log q)^{\gamma}$. Recall the truncation operator $T$ from (\ref{eq:Truncation operator}). Noting $E(J_j)=O(1)$ from Lemma \ref{lem:rosenthal lemma}, we have 
\begin{eqnarray*}
P(|\sum_{j=1}^q T_{\theta}(J_j)-E(T_{\theta}(J_j)) |\geq \phi) \leq { q \over \phi^2} \max_j E(T_{\theta}(J_j)^2) = O(\theta q/\phi^2) =o(1),
\end{eqnarray*}
where $\phi = q/\log q$, and  $$ \max_j P(J_j\geq \theta) \leq \max_j P(E(|\tilde{S}_{2k_0mj}-\tilde{S}_{2k_0m(j-1)}|^{2}|\bar{a}_{2k_0(j-1)}, \bar{a}_{2k_0j}) \geq 2k_0lm)=o(q^{-1}),$$ \noindent from (\ref{eq:1st term small}), (\ref{eq:vj 4th term diff small}) and (\ref{eq:vja-evja small 2}). Thus $P(|\sum_{j=1}^{q}J_j - \sum_{j=1}^{q}E(J_j)| \geq \phi) \to 0$ which is a restatement of (\ref{eq:Lgamma1}).
\end{proof}

\begin{proof} of Proposition \ref{prop:choice of l}. We showed in Proposition \ref{prop:Lgamma} that $$P(cqm^{\gamma/2} \leq L_{\gamma} \leq Cqm^{\gamma/2})\to 1,$$
for some constants $c$ and $C$. Let $l$ be as given in (\ref{eq:choice of l}). Let $S=\{0,l,2l,\cdots \}$. Proposition \ref{prop:lower bound on Var Yj'} and Proposition \ref{prop:vja} show that, for some constants $c$ and $C$,
\begin{eqnarray*}
P( c lk_0m \leq \min_{i \in S}\rho_*\left( Var \left(\displaystyle \sum_{j=i}^{i+l-1}Y_j^{a}\right)\right) \leq \max_{i \in S}\rho^* \left( Var \left(\displaystyle \sum_{j=i}^{i+l-1}Y_j^{a}\right)\right) \leq  C  lk_0m ) \to 1.
\end{eqnarray*}
We choose $\eta_k=k l $ and $s \asymp q/l$. Starting with the conditional block sum process $Y_j^{a}$ for $0 \leq j \leq q-1$, this choice of $\eta_k$ satisfies (\ref{eq:zaitsev condition 1}) for a given $a$ with probability going to 1. The other condition, (\ref{eq:zaitsev condition 2}) can be easily verified for such a choice of $\eta$-sequence using ideas similar to the proof of Proposition \ref{prop:Lgamma}. We skip the details of that derivation. 
\end{proof}









\section{Some Useful Results}
\label{sec:sul}
\begin{lemma}\label{lem:important inequality} 
Let $p< \gamma$. Assume (2.A). Then $\sup_i E \min \{ |X_i|^\gamma n^{-\gamma/p}, 1 \} = o(n^{-1}).$
\end{lemma}

\begin{proof}
Choose $k_n = \lfloor 2 (\log n) / ( (p+\gamma) \log 2) \rfloor$. Then $n = o(2^{\gamma k_n})$ and $2^{p k_n} = o(n)$. Let $Z = |X_i| n^{-1/p}$. The lemma follows from
\begin{eqnarray*}
 E (\min \{ Z^\gamma, 1 \})
 &\le& P(Z\ge 1) + \sum_{k=0}^{k_n} 2^{-k\gamma} P(2^{-1-k} \le Z < 2^{-k}) 
 + 2^{-\gamma(k_n+1)} \cr
 &\le& E(Z^p {\bf 1}_{Z\ge 1}) + \sum_{k=0}^{k_n} 2^{p(k+1)-k\gamma} E(Z^p {\bf 1}_{ Z \ge 2^{-1-k}})
 + 2^{-\gamma(k_n+1)} = o(n^{-1}),
\end{eqnarray*}
in view of the uniform integrability condition (2.A) and $n^{1/2}/2^{k_n} \to \infty$. 
\end{proof}

\begin{lemma}\label{lem:truncated deltajy}
The functional dependence measures defined on the truncated process $(X_i^{\oplus})$ and the $m$-dependent process $(\tilde{X}_i)$, satisfy $\tilde{\delta}_{j,\gamma} \leq \delta^{\oplus}_{j,\gamma} \leq 2 n^{1/p-1/\gamma} t_n^{1-p/\gamma} \delta_{j,p}^{p/\gamma}.$
\end{lemma}
\begin{proof}

\noindent Since the truncation operator $T$ is Lipschitz continuous, 
\begin{eqnarray*}\label{eq:proof of lemma}
(\delta^{\oplus}_{j,\gamma})^{\gamma} &= &\sup_{i} E( | T_{t_nn^{1/p}}(X_i)- T_{t_nn^{1/p}}(X_{i,(i-j)}) |^{\gamma} ) \\ \nonumber
&=& n^{\gamma/p} t_n^{\gamma}\sup_{i} E\left(\left|\min \left(2, \left|\frac{X_i- X_{i,(i-j)}}{t_n n^{1/p}}\right| \right)\right|^{\gamma}\right) \leq 2^{\gamma} n^{\gamma/p-1}t_n^{\gamma-p}\delta_{j,p}^p.
\end{eqnarray*}
The first inequality $\tilde{\delta}_{j,\gamma} \leq \delta^{\oplus}_{j,\gamma}$ follows from (\ref{eq:fdm inequality}). 
\end{proof}


\begin{lemma}
\label{lem:rosenthal lemma}
\underline{Rosenthal Type Moment Bound} 
Recall (\ref{eq:tn part1}) and (\ref{eq:tn part3 dup}) for $t_n$. Assume (\ref{eq:first condition}), (\ref{eq:second condition}), (\ref{eq:third condition}) along with (\ref{eq:A condition}) on $A$ related to the restriction on $\Theta_{i,p}$ as mentioned in (\ref{eq:form of thetaip}). Moreover, assume $m=\lfloor n^{L}t_n^k \rfloor$ with $k$ satisfying $k<(\gamma/2-1)^{-1}(\gamma-p)$. Then, we have
\begin{eqnarray}\label{eq:Sm bound}
\max_t E(\max_{1 \leq l \leq m} |\tilde{R}_{t,l}|^{\gamma})=O(m^{\gamma/2}).
\end{eqnarray}

\end{lemma}

\begin{proof}
\noindent Since the functional dependence measure is defined in an uniform manner, we can ignore the $\max_t$ in (\ref{eq:Sm bound}) and use the Rosenthal-type inequality for stationary processes in Liu, Xiao and Wu (2013, \cite{MR3114713}). By \cite{MR3114713}, there is a constant $c$, depending only on $\gamma$, such that
\begin{eqnarray*}\label{eq:rosenthal ineq}
\|\displaystyle\max_{1\leq l\leq m}|\tilde{R}_{t,l}| \|_{\gamma} & \leq & c m^{1/2} [\displaystyle\sum_{j=1}^{m} \tilde{\delta}_{j,2}+ \displaystyle\sum_{j=1+m}^{\infty} \tilde{\delta}_{j,\gamma}+ \sup_{i}\| T_{t_nn^{1/p}}(X_i) \| ] \\ \nonumber
& \hspace{0.05 in} &+ c m^{1/\gamma} [\displaystyle\sum_{j=1}^{m} j^{1/2-1/\gamma}\tilde{\delta}_{j,\gamma}+ \sup_{i}\| T_{t_nn^{1/p}}(X_i) \|_{\gamma} ] \\ \nonumber
& \leq & c(I+II+III+IV),
\end{eqnarray*}
where 
\begin{eqnarray*}\label{eq:three terms}
I&=& m^{1/2}\displaystyle\sum_{j=1}^{m}\tilde{\delta}_{j,2}+m^{1/2}\| X_1\|_2, \\ \nonumber
II&=& m^{1/2}\sum_{j=m+1}^{\infty}\tilde{\delta}_{j,\gamma}, \quad III=m^{1/\gamma} \displaystyle\sum_{j=1}^{\infty} j^{1/2-1/\gamma} \tilde{\delta}_{j,\gamma}, \\ \nonumber
IV&=& m^{1/\gamma} \sup_{i} \| T_{t_nn^{1/p}}(X_i) \|_{\gamma}.
\end{eqnarray*}

\noindent For the first term $I$, since $\sum_{j=1}^{\infty}\delta_{j,2}+\sup_{i} \|X_i \|_2 \leq 2 \Theta_{0,2}$ and $\tilde{\delta}_{j,2} \leq \delta_{j,2} $, we have $ I = O(m^{1/2}).$ Starting with $II$, we apply Lemma \ref{lem:truncated deltajy} to obtain 
\begin{eqnarray*}\label{eq:Simplify II} 
II=m^{1/2}\sum_{j=m+1}^{\infty}\tilde{\delta}_{j,\gamma}
\lesssim m^{1/2} n^{1/p-1/\gamma}t_n^{1-p/\gamma}  \sum_{j=m+1}^{\infty}\delta_{j,p}^{p/\gamma}.
\end{eqnarray*}
The rest follows from the derivation in (\ref{eq:third condition simplify}) and (\ref{eq:equation 3}). For the third term, we have
\begin{eqnarray}\label{eq:third term simplify}
III &\lesssim &   m^{1/\gamma}n^{1/p-1/\gamma}t_n^{1-p/\gamma}\sum_{j=1}^{m} j^{1/2-1/\gamma}\delta_{j,p}^{p/\gamma}  \\ \nonumber
 & \leq& m^{1/\gamma}n^{1/p-1/\gamma}t_n^{1-p/\gamma} \sum_{l=1}^{\lfloor{ \log_2m}\rfloor+1} \sum_{j=2^{l-1}}^{2^l-1} j^{1/2-1/\gamma} \delta_{j,p}^{p/\gamma} \\ \nonumber
 & \leq & m^{1/\gamma} n^{1/p-1/\gamma}t_n^{1-p/\gamma} \sum_{l=1}^{\lfloor{\log_2 m}\rfloor+1} 2^ {l (3/2-1/\gamma- p/\gamma)} O(2^{-l\chi p/\gamma} l ^{-Ap/\gamma}).
\end{eqnarray}
\noindent Recall the definition of $\chi_0$ from (\ref{eq:chi0}). If $\chi\leq \chi_0$, then our solution for $\gamma$ satisfies 
$$3/2-1/\gamma- (\chi+1)p/\gamma \geq 0,$$ \noindent with equality holding only for $\chi=\chi_0$. Hence, if $\chi<\chi_0$, we have 
\begin{eqnarray*}\label{eq:III simplify part 1}
m^{-1/2}III = m^{1- (\chi+1)p/\gamma} n^{1/p-1/\gamma}t_n^{1-p/\gamma} (\log n)^{-Ap/\gamma} O(1) =o(1),
\end{eqnarray*}
\noindent from (\ref{eq:equation 3}), (\ref{eq:define m dup}) and (\ref{eq:tn part3 dup}). If $\chi = \chi_0$, since $A>\gamma/p$ from (\ref{eq:A condition}) [The lower bound for $A$ there is just $2\gamma/p$ as mentioned in (\ref{eq:equation 1})], we have
\begin{eqnarray}\label{eq:III simplify part 2}
m^{-1/2}III= m^{1/\gamma-1/2} n^{1/p-1/\gamma}t_n^{1-p/\gamma} O(1) = o(1),
\end{eqnarray}
\noindent since (\ref{eq:equation 2}) is true. Also for the case of $\chi>\chi_0$ in the proof of Theorem \ref{th:theorem 2}, the way we define our three conditions in (\ref{eq:new set})
the new solution also satisfy $\gamma'=2(1+p+p\chi)/3$ and thus (\ref{eq:III simplify part 2}) holds. For the fourth term $IV$, we use (\ref{eq:tn part1}) to derive 
\begin{eqnarray}\label{eq:third term}
m^{-\gamma/2} IV^{\gamma} &=&m^{1-\gamma/2} \sup_{i} \|T_{t_nn^{1/p}}(X_i) \|^{\gamma}  \\
&\leq& m^{1-\gamma/2}t_n^{\gamma}n^{\gamma/p}   \sup_{i} E \left(\min \{ \frac{|X_i|^\gamma}{t_n^{\gamma}n^{\gamma/p}}, 1\}\right) \nonumber \\
&=&m^{1-\gamma/2}t_n^{\gamma}n^{\gamma/p-1}o(1) = o(1), \nonumber
\end{eqnarray}
in the light of (\ref{eq:equation 2}). 
\end{proof}





\end{document}